\documentclass{birkjour_t2}
\usepackage{amsmath}
\usepackage{tkz-euclide}
\usetkzobj{all}
\tikzset{elegant/.style={smooth,thick,samples=50,cyan}}
\tikzset{eaxis/.style={->,>=stealth}}
\usepackage{color,tikz}
\usetikzlibrary{decorations.pathreplacing,calc}
\usepackage{rotating}
\tikzset{liltext/.style={font=\tiny}}
\usepackage{amsfonts}
\usepackage{amssymb}
\usepackage[english]{babel}
\usepackage{color}
\usepackage{graphicx}
\usepackage{lmodern}
\usepackage{amsthm}
\usepackage{mathrsfs}
\usepackage{microtype}
\usepackage{mathscinet}
\usepackage{enumitem}
\usepackage[cal=boondoxo,bb=ams]{mathalfa}
\usepackage{hyperref}
\hypersetup{hidelinks}

\renewcommand{\theequation}{\arabic{section}.\arabic{equation}}

\newtheorem{thm}{Theorem}[section]
\newtheorem{prop}{Proposition}[section]
\newtheorem{lem}{Lemma}[section]

\newtheorem{rem}{Remark}[section]

\newcommand{\ml}{\mathcal}
\newcommand{\mb}{\mathbb}
\newcommand{\dd}{\mathrm{d}}
\newcommand{\ee}{\mathrm{e}}

\DeclareMathOperator{\divv}{div}

\DeclareMathOperator{\lin}{lin}
\DeclareMathOperator{\non}{non}

\begin{document}

%
%
%
%
%
%
%
%

\title[Semilinear strongly damped wave equations in an exterior domain]
 {A global existence result for two-dimensional semilinear strongly damped wave equation with mixed nonlinearity in an exterior domain}

\author[W. Chen]{Wenhui Chen}
\address{Institute of Applied Analysis, Faculty of Mathematics and Computer Science\\
	 Technical University Bergakademie Freiberg\\
	  Pr\"{u}ferstra{\ss}e 9\\
	   09596 Freiberg\\
	    Germany}
\email{wenhui.chen.math@gmail.com}

\author[A.Z. Fino]{Ahmad Z. Fino}

\address{Department of Mathematics, Faculty of Sciences\\
	 Lebanese University\\
	P.O. Box 826\\
	Tripoli\\
	Lebanon}
\email{ahmad.fino01@gmail.com; afino@ul.edu.lb}

\subjclass{Primary 35A01; Secondary 35L20, 35L05}

\keywords{{Semilinear damped wave equation, exterior domain, strong damping, mixed nonlinearity, global existence.}
}
\date{October 25, 2019}

\begin{abstract} 
We study two-dimensional semilinear strongly damped wave equation with mixed nonlinearity $|u|^p+|u_t|^q$ in an exterior domain, where $p,q>1$. Assuming the smallness of initial data in exponentially weighted spaces and some conditions on powers of nonlinearity, we prove global (in time) existence of small data energy solution with suitable higher regularity by using a weighted energy method.
\end{abstract}

\maketitle

\section{Introduction}
In this paper, we consider the following initial boundary value problem for two-dimensional semilinear strongly damped wave equation with mixed nonlinearity:
\begin{equation}\label{Eq.Semi.BVP}
\begin{cases}
u_{tt}-\Delta u -\Delta u_t =|u|^p+|u_t|^q, &x\in \Omega,\,t>0,\\
u(0,x)=  u_0(x),\,\,u_t(0,x)=  u_1(x),&x\in \Omega,\\
u=0,&x\in \partial\Omega,\,t>0,
\end{cases}
\end{equation} 
with $p,q>1$, where $\Omega\subset\mathbb{R}^2$ is an exterior domain with a compact smooth boundary $\partial\Omega$. Without loss of generality, we may assume that $0\notin\overline{\Omega}$.

 In the whole space $\mb{R}^n$, the Cauchy problem for the linear strongly damped wave equation can be modeled by
 \begin{equation}\label{Eq.Linear.SDW}
 \begin{cases}
 u_{tt}-\Delta u-\Delta u_t=0,&x\in\mathbb{R}^n,\,t>0,\\
 u(0,x)=  u_0(x),\,\,u_t(0,x)=  u_1(x),&x\in \mathbb{R}^n,
 \end{cases}
 \end{equation}
 where $n\geqslant 1$. The authors of \cite{Ponce1985} and \cite{Shibata2000} investigated some $L^p-L^q$ estimates away of the conjugate line for the Cauchy problem \eqref{Eq.Linear.SDW}. Recently, asymptotic profiles of solutions in $L^2$-norm to the Cauchy problem \eqref{Eq.Linear.SDW} under a framework of weighted $L^1$ data were derived in \cite{Ikehata2014}. Concerning asymptotic profiles of solutions to the corresponding abstract form of strongly damped waves were derived in  \cite{IkehataTodorovaYordanov2013}. The authors of \cite{DAbbiccoReissig2014} obtained $(L^2\cap L^1)-L^2$ estimate for \eqref{Eq.Linear.SDW} by using the partial Fourier transform and WKB analysis. By employing some energy estimates with suitable regularities, additionally, they proved global (in time) existence of small data solution to the semilinear Cauchy problem. More precisely, \cite{DAbbiccoReissig2014} considered the semilinear Cauchy problem with power nonlinearity $|u|^p$, namely,
 \begin{equation}\label{Eq.Semilinear.SDW}
 \begin{cases}
 u_{tt}-\Delta u-\Delta u_t=|u|^p,&x\in\mathbb{R}^n,\,t>0,\\
 u(0,x)=  u_0(x),\,\,u_t(0,x)=  u_1(x),&x\in \mathbb{R}^n,
 \end{cases}
 \end{equation}
 where $p>1$ and $n\geqslant 2$. The authors of \cite{DAbbiccoReissig2014} proved global existence results for $n\geqslant2$ if $p>1+3/(n-1)$ and $p\in[2,n/(n-4)]$. Besides, applying the test function method, the result for nonexistence of global (in time) solutions has been proved providing that $1<p\leqslant 1+2/(n-1)$. Here, we refer to Theorem 4.2 in \cite{D'AmbrosioLucente2003}.
 
 Let us turn to initial boundary value problem with an exterior domain.  To the best of the authors' knowledge, there exist few results on the exterior problem for strongly damped wave equation.  We refer to  \cite{Ikehata2001} for some estimates of solutions to the linearized exterior problem by using suitable energy method. Later, \cite{Ikehata-Inoue2008} proved global existence result for the semilinear strongly damped wave equation with $|u|^p$ in 2D when $p>6$. To be specific, there is a uniquely determined energy solution
 \begin{align*}
 u\in\ml{C}\left([0,\infty),H_0^1(\Omega)\right)\cap \ml{C}^1\left([0,\infty),L^2(\Omega)\right)
 \end{align*}
 to the exterior problem
 \begin{equation}\label{Eq.Semi.|u|^p.BVP}
 \begin{cases}
 u_{tt}-\Delta u -\Delta u_t =|u|^p, &x\in \Omega,\,t>0,\\
 u(0,x)=  u_0(x),\,\,u_t(0,x)=  u_1(x),&x\in \Omega,\\
 u=0,&x\in \partial\Omega,\,t>0,
 \end{cases}
 \end{equation} 
 where $\Omega\subset\mathbb{R}^2$ is an exterior domain with a compact smooth boundary, for small initial data taken from weighted energy space and  $p>6$. On the other hand, the blow-up  result for the exterior problem \eqref{Eq.Semi.|u|^p.BVP} for all dimensions has been recently derived by \cite{Fino}. Considering the semilinear exterior problem for strongly damped wave with nonlinearity of derivative-type $|u_t|^q$ or nonlinearity of mixed-type $|u|^p+|u_t|^q$, the authors of \cite{Chen-Fino2019} proved local (in time) existence of mild solution by energy estimates associated with the Banach fixed-point theorem and blow-up of solutions by the test function method in all dimension $n\geqslant 1$. For the other studied of damped wave equations in an exterior domain, we refer to \cite{Ikehata2003,One2003,Ikehata2004,Hayashi-Kaikina-Naumkin2004,Ogawa-Takeda2009,Lai-Yin2017,Fino-Ibrahim-Wehbe2017,Sobajima2019,Dabbicco-Takeda-Ikehata2019} and references therein. So far the global (in time) existence of energy solution with higher regularity for semilinear strongly damped wave equations in an exterior domain is still unknown. Hence, in the present paper, we will give the answer to it by constructing some exponentially weighted energy estimates of higher-order.
 
 Our main approach in proving global existence result is a exponentially weighted energy method, which has been introduced firstly in the pioneering work \cite{Todorova-Yordanov2001} and later in \cite{Ikehata-Tanizawa2005} with certain exponentially weighted spaces. In some sense, this method strongly relies on the choice of the weighted function. In this paper, motivated by \cite{Ikehata-Inoue2008}, we introduce the weighted function by the following way:
 \begin{align}\label{Weighted.Fun}
 \psi(t,x)\doteq\frac{1}{\rho(1+t)^\rho}+\frac{|x|^2}{2(1+t)^{2+\rho}}
 \end{align}
 with a constant $\rho>\rho_0$, where
 \begin{align*}
 \rho_0\doteq\frac{-3+\sqrt{73}}{4}\approx1.386
 \end{align*}
  is the positive root of the quadratic equation 
 \begin{align*}
 2\rho^2_0+3\rho_0-8=0.
 \end{align*} 
 In order to prove existence of higher-order energy solution, we have to construct the exponentially weighted estimates for higher-order energy to \eqref{Eq.Semi.BVP}, that is
 \begin{align*}
 \left\|\ee^{\psi(t,\cdot)}(\partial_t,\nabla)u(t,\cdot)\right\|_{L^2}^2\quad\mbox{and}\quad \left\|\ee^{\psi(t,\cdot)}\nabla(\partial_t,\nabla)u(t,\cdot)\right\|_{L^2}^2.
 \end{align*}
  We would like to remark that the estimate of the second energy mentioned in the above is more complex than those of the first energy, which will be shown in Propositions \ref{prop1} and \ref{prop2} in Section \ref{Sec.Weighted.Nonlinear}.

  Let us point out that the study of the exterior problem with $|u|^p+|u_t|^q$ is not simply a generalization of what happens for the exterior problem with $|u|^p$ shown in \cite{Ikehata-Inoue2008}. On one hand, the application of the Gagliardo-Nirenberg inequality with a weighted function (e.g. Lemma 2.3 in \cite{Ikehata-Inoue2008}) allows us to estimate the nonlinear term $|u_t|^q$ in a weighted $L^q$ space by its gradient in $L^2$ space and it in weighted $L^2$ space. Thus, we need to control a new energy of nonlinear exterior problems. On the other hand, the interplay between the power nonlinearity $|u|^p$ and nonlinearity of derivative-type $|u_t|^q$ should be considered, which does not happen for the nonlinear problem carrying $|u|^p$ only.

 Before stating our global existence result, we show some notations will be used in this paper.

  By direct computations, the next properties for the function $\psi(t,x)$ are fulfilled:
 \begin{align*}
 \psi_t(t,x)<0,\quad\Delta\psi(t,x)=\frac{2}{(1+t)^{2+\rho}}\quad\mbox{and}\quad-\psi_t(t,x)\leqslant\frac{C_{\rho}}{1+t}\psi(t,x),
 \end{align*}
 where the positive constant $C_{\rho}$ independent of $x$ and $t$. Furthermore, it holds that
 \begin{align}\label{est15}
 |\nabla\psi(t,x)|^2-\psi_t(t,x)|\nabla\psi(t,x)|^2-|\psi_t(t,x)|^2\leqslant 0.
 \end{align}
 We assume $\varepsilon>0$ be an auxiliary constant satisfying
 \begin{equation}\label{epsilon}
 \frac{4\rho+14}{(2+\rho)(2\rho+3)}\leqslant\varepsilon<1.
 \end{equation} 
 
 Next, we define the space-dependent function (see Theorem 1.1 in \cite{Ikehata2001})
 \begin{align*}
 d(x)\doteq |x|\log(B|x|)
 \end{align*} 
 with a positive constant $B$ such that $\inf\limits_{x\in\Omega}|x|\geqslant 2/B>0$.
 
 Finally, we introduce a norm for initial data such that
 \begin{align*}
 \ml{J}[u_0,u_1]\doteq\sum\limits_{j=0,1}\left(\|u_j\|_{L^2}^2+\|\nabla u_j\|_{L^2}^2\right)+\|\Delta u_0\|_{L^2}^2+\|d(\cdot)\Delta u_0\|_{L^2}^2+\|d(\cdot)u_1\|_{L^2}^2+I_{\mathrm{exp}}[u_0,u_1],
 \end{align*}
 where the exponentially weighted norm for initial data is define by
 \begin{align}\label{Weighted.Data}
 I_{\mathrm{exp}}[u_0,u_1]\doteq\int_{\Omega}\ee^{2\psi(0,x)}\left(|\nabla u_1(x)|^2+|\Delta u_0(x)|^2+|u_1(x)|^2+|\nabla u_0(x)|^2\right)\dd x.
 \end{align}
 
 Let us state our main result.
\begin{thm}\label{Thm.GESDS} Let us assume 
	\begin{align}\label{Condition.p,q}
	p>6+2\rho_0\quad\mbox{and}\quad q>6+2\rho_0.
	\end{align}
	Then, there exists a constant $\varepsilon_0>0$ such that for any
	\begin{align*}
	(u_0,u_1)\in \left(H^2(\Omega)\cap H^1_0(\Omega)\right)\times H^1(\Omega)
	\end{align*}
with $\ml{J}[u_0,u_1]\leqslant\varepsilon_0$, there is a uniquely determined energy solution of higher-order 
	\begin{align*}
	u\in\ml{C}\left([0,\infty),H^2(\Omega)\cap H^1_0(\Omega)\right)\cap\ml{C}^1\left([0,\infty),H^1(\Omega)\right)
	\end{align*}
	to \eqref{Eq.Semi.BVP}. Furthermore, the solution satisfies the following estimates:
	\begin{align*}
	\|u(t,\cdot)\|_{L^2}^2&\leqslant C \ml{J}[u_0,u_1],
	\end{align*}
	\begin{align*}
	\|u_t(t,\cdot)\|_{L^2}^2+\|\nabla u(t,\cdot)\|_{L^2}^2+\|\nabla u_t(t,\cdot)\|_{L^2}^2+\|\Delta u(t,\cdot)\|_{L^2}^2\leqslant C(1+t)^{-1}\ml{J}[u_0,u_1],
	\end{align*}
	\begin{align*}
	\left\|\ee^{\psi(t,\cdot)}u_t(t,\cdot)\right\|_{L^2}^2+\left\|\ee^{\psi(t,\cdot)}\nabla u(t,\cdot)\right\|_{L^2}^2+\left\|\ee^{\psi(t,\cdot)}\nabla u_t(t,\cdot)\right\|_{L^2}^2+\left\|\ee^{\psi(t,\cdot)}\Delta u(t,\cdot)\right\|_{L^2}^2\leqslant C\ml{J}[u_0,u_1],
	\end{align*}
	for any $t\geqslant 0$.
\end{thm}
\begin{rem}
Here, we emphasize that  one cannot not compare the hereinbefore proposed result with those result of the previous research \cite{Ikehata-Inoue2008}. First of all, the nonlinear term of what we treat is different from \cite{Ikehata-Inoue2008}. What's more, as mentioned before, the authors of \cite{Ikehata-Inoue2008} proved global (in time) existence for classical energy solution
\begin{align*}
u\in\ml{C}\left([0,\infty),H^1_0(\Omega)\right)\cap\ml{C}^1\left([0,\infty),L^2(\Omega)\right)
\end{align*}
for semilinear strongly damped wave equation with power nonlinearity $|u|^p$, where initial data are taken from $(H^2(\Omega)\cap H_0^1(\Omega))\times L^2(\Omega)$. What we do in this paper is to derive the global (in time) existence of energy solution of higher-order. It seems reasonable to have a stronger condition on $p$.
\end{rem}

\begin{rem}
Considering the weighted function $\psi(t,x)$ defined in \eqref{Weighted.Fun}, roughly speaking, the reason for us to consider $\rho>\rho_0$ is in the derivation of weighted estimates for higher-order energy. We will see later in Proposition \ref{prop1}.
\end{rem}

\begin{rem}
The critical curve $\Upsilon(n)$ of the $n$-dimensional semilinear strongly damped wave equation in an exterior domain carrying mixed nonlinear term $|u|^p+|u_t|^q$ is still open. Here, the critical curve means that if a pair of exponents $p$ and $q$ are above the curve $\Upsilon(n)$, there exists global (in time) small data solution; on the contrary, if the exponents are on or below the curve $\Upsilon(n)$, every local (in time) solutions blows up in finite time even with small data. However, under the assumption of initial data taken from energy space with suitable higher regularity, the blow-up of solutions with suitable condition on the exponent is still unknown. For this reason, so far we cannot tell whether or not the restriction \eqref{Condition.p,q} is the critical curve in 2D.
\end{rem}

 The remaining part of the present paper is organized as follows. In Section \ref{Sec.Linear}, we derive energy estimates for the corresponding linear homogeneous problem to \eqref{Eq.Semi.BVP}. In Section \ref{Sec.Weighted.Nonlinear}, some exponentially weighted $L^2$ estimates for nonlinear strongly damped wave equation are obtained. In Section \ref{Sec.Proof.Thm}, we prove Theorem \ref{Thm.GESDS}. Finally, final remark in Section \ref{Sec.Final.Remark} completes the paper.

\section{Energy estimates for linear homogeneous strongly damped wave equation}\label{Sec.Linear}
\setcounter{equation}{0}
In the section, we are concerned with energy estimates for the corresponding linearized equation to \eqref{Eq.Semi.BVP}, namely,
\begin{equation}\label{eq1}
\begin{cases}
u_{tt}-\Delta u -\Delta u_t =0, &x\in \Omega,\,t>0,\\
u(0,x)=  u_0(x),\,u_t(0,x)=  u_1(x),&x\in \Omega,\\
u=0,&x\in \partial\Omega,\,t>0.
\end{cases}
\end{equation} 

To do this, let us define an energy containing higher-order derivative of solutions for \eqref{eq1} firstly
\begin{align*}
E[u](t)\doteq \frac{1}{2}\left(\|u_t(t,\cdot)\|_{L^2}^2+\|\nabla u(t,\cdot)\|^2_{L^2}+\|\nabla u_t(t,\cdot)\|_{L^2}^2+\|\Delta u(t,\cdot)\|_{L^2}^2\right).
\end{align*}

We found that this energy is in a higher-order sense, which is different from total energy defined in \cite{Ikehata2001}. Therefore, we need to derive a new estimate of the energy $E[u](t)$ in the next lemma.
\begin{lem}\label{Lemma.Linear}
	Let us assume
	\begin{align*}
	(u_0,u_1)\in \left(H^2(\Omega)\cap H^1_0(\Omega)\right)\times  H^1(\Omega)
	\end{align*}
	 satisfying $\left\|d(\cdot)(u_1-\Delta u_0)\right\|_{L^2}<\infty$. Then, the following energy estimate for \eqref{eq1} holds:
	\begin{align}\label{est23}
	E[u](t)\leqslant  (1+t)^{-1}I_2[u_0,u_1],
	\end{align}
	where the constant $I_2[u_0,u_1]$ with respect to initial data will be defined in \eqref{est38} later.
\end{lem}
\begin{proof}
First of all, multiplying the equation in \eqref{eq1} by $u_t$ and integrating the resulting identity over $\Omega$, one gets
\begin{align*}
\frac{1}{2}\frac{\dd}{\dd t}\int_\Omega \left(|u_{t}(t,x)|^2+|\nabla u(t,x)|^2\right)\dd x+\int_\Omega |\nabla u_{t}(t,x)|^2\,\dd x=0.
\end{align*}
To construct the energy, let us integrate the above equation over $[0,t]$ such that
\begin{align}\label{est1}
\frac{1}{2}\left(\|u_{t}(t,\cdot)\|_{L^2}^2+\|\nabla u(t,\cdot)\|_{L^2}^2\right)+\int_0^t\|\nabla u_{t}(s,\cdot)\|_{L^2}^2\,\dd s=\frac{1}{2}\left(\|u_1\|^2_{L^2}+\|\nabla u_0\|^2_{L^2}\right).
\end{align} 

With the aim of deriving higher-order energy, we multiply the equation in \eqref{eq1} by $\Delta u_t$ and integrate over $\Omega$ to have
\begin{align}\label{EQQQ05}
\frac{1}{2}\frac{\dd}{\dd t}\int_\Omega\left( |\nabla u_t(t,x)|^2+ |\Delta u(t,x)|^2\right)\dd x+\int_\Omega |\Delta u_{t}(t,x)|^2\,\dd x=0.
\end{align}
Integrating the resulting equation over $[0,t]$ leads to
\begin{equation}\label{est2}
\frac{1}{2}\left(\|\nabla u_t(t,\cdot)\|_{L^2}^2+\|\Delta u(t,\cdot)\|_{L^2}^2\right)+\int_0^t\|\Delta u_{t}(s,\cdot)\|^2_{L^2}\,\dd s=\frac{1}{2}\left(\|\nabla u_1\|_{L^2}^2+\|\Delta u_0\|_{L^2}^2\right).
\end{equation} 

Similarly as the previous two steps, we multiply the equation in \eqref{eq1} by $\Delta u$ and integrate over $\Omega$ as well as the interval $[0,t]$. It implies
\begin{align*}
-\frac{\dd}{\dd t}\int_\Omega \nabla u_{t}(t,x)\cdot\nabla u(t,x)\,\dd x+\int_\Omega |\nabla u_t(t,x)|^2\,\dd x -\int_\Omega |\Delta u(t,x)|^2\,\dd x-\frac{1}{2}\frac{\dd}{\dd t}\int_\Omega |\Delta u(t,x)|^2\,\dd x=0
\end{align*}
and then,
\begin{align*}
\int_0^t\int_\Omega |\Delta u(s,x)|^2\,\dd x\,\dd s+\frac{1}{2}\int_\Omega |\Delta u(t,x)|^2\,\dd x&=\int_0^t\int_\Omega |\nabla u_t(s,x)|^2\,\dd x\,\dd s-\int_\Omega \nabla u_{t}(t,x)\cdot \nabla u(t,x)\,\dd x\\
&\quad+\frac{1}{2}\int_\Omega |\Delta u_0(x)|^2\,\dd x+\int_\Omega \nabla u_{1}(x)\cdot \nabla u_0(x)\,\dd x.
\end{align*}
By employing Young's inequality, we may conclude that
\begin{align}\label{est3}
\int_0^t\|\Delta u(s,\cdot)\|_{L^2}^2\,\dd s+\frac{1}{2}\|\Delta u(t,\cdot)\|_{L^2}^2&\leqslant \int_0^t\|\nabla u_t(s,\cdot)\|_{L^2}^2\,\dd s+\frac{1}{2}\|\nabla u_{t}(t,\cdot)\|^2_{L^2}+\frac{1}{2}\|\nabla u(t,\cdot)\|^2_{L^2}\notag\\
&\quad+\frac{1}{2}\|\Delta u_0\|^2_{L^2}+\frac{1}{2}\|\nabla u_{1}\|^2_{L^2}+\frac{1}{2}\|\nabla u_0\|^2_{L^2}\notag\\
&\leqslant \|\Delta u_0\|^2_{L^2}+\|\nabla u_{1}\|^2_{L^2}+\|\nabla u_0\|^2_{L^2}+\frac{1}{2}\| u_{1}\|^2_{L^2},
\end{align}
where we have used the equations \eqref{est1} and \eqref{est2}.

According to Theorem 1.1 in \cite{Ikehata2001}, we know under our assumption of initial data in Lemma \ref{Lemma.Linear}, the integration with respect to time variable of classical energy is bounded such that
\begin{align}\label{est5}
\int_0^t \left(\|u_t(s,\cdot)\|_{L^2}^2+\|\nabla u(s,\cdot)\|_{L^2}^2\right)\dd s\leqslant I_0[u_0,u_1],
\end{align}
where the constant $I_0[u_0,u_1]$ is denoted by
\begin{align*}
I_0[u_0,u_1]\doteq 2\|u_0\|^2_{L^2}+\|u_1\|_{L^2}^2+\frac{1}{2}\|\nabla u_0\|_{L^2}^2+3C_0\left\|d(\cdot)(u_1-\Delta u_0)\right\|_{L^2}^2,
\end{align*}
with Hardy's constant $C_0>0$. On the other hand, from \eqref{est1} and \eqref{est3} one observes
\begin{equation}\label{est6}
\int_0^t \left(\|\nabla u_t(s,\cdot)\|_{L^2}^2+\|\Delta u(s,\cdot)\|_{L^2}^2\right)\mathrm{d}s\leqslant I_1[u_0,u_1],
\end{equation} 
where the constant $I_1[u_0,u_1]$ is defined by
\begin{align*}
I_1[u_0,u_1]\doteq \|\Delta u_0\|^2_{L^2}+\|\nabla u_{1}\|^2_{L^2}+\frac{3}{2}\|\nabla u_0\|^2_{L^2}+\| u_{1}\|^2_{L^2}.
\end{align*}
Furthermore, we notice from \eqref{EQQQ05} that
\begin{align*}
&\frac{\dd}{\dd t}\left((1+t)\left(\|\nabla u_t(t,\cdot)\|_{L^2}^2+\|\Delta u(t,\cdot)\|_{L^2}^2\right)\right)\\
&=\|\nabla u_t(t,\cdot)\|_{L^2}^2+\|\Delta u(t,\cdot)\|_{L^2}^2+(1+t)\frac{\dd}{\dd t}\left(\|\nabla u_t(t,\cdot)\|_{L^2}^2+\|\Delta u(t,\cdot)\|_{L^2}^2\right)\\
&\leqslant \|\nabla u_t(t,\cdot)\|_{L^2}^2+\|\Delta u(t,\cdot)\|_{L^2}^2.
\end{align*}
The next decay estimate for higher-order energy yields immediately by integrating the above derived inequality over $[0,t]$: 
\begin{align}\label{EQQQ06}
\|\nabla u_t(t,\cdot)\|_{L^2}^2+\|\Delta u(t,\cdot)\|_{L^2}^2\leqslant (1+t)^{-1}\left(I_1[u_0,u_1]+\|\nabla u_1\|_{L^2}^2+\|\Delta u_0\|_{L^2}^2\right).
\end{align}
Here, we applied \eqref{est6}. The estimate \eqref{EQQQ06} will be used later.

From the derived estimates \eqref{est1} as well as \eqref{est2}, we see
\begin{align*}
\frac{\dd}{\dd t}E[u](t)=-\left(\|\nabla u_t(t,\cdot)\|_{L^2}^2+\|\Delta u_t(t,\cdot)\|_{L^2}^2\right)\leqslant 0,
\end{align*}
which implies
\begin{align*}
\frac{\dd}{\dd t}\left((1+t)E[u](t)\right)=E[u](t)+(1+t)\frac{\dd}{\dd t}E[u](t)\leqslant E[u](t).
\end{align*}
We integrate the above inequality over $[0,t]$ to have
\begin{align}\label{est4}
(1+t)E[u](t)\leqslant \int_0^t E[u](s)\,\dd s+E[u](0).
\end{align}

Finally, taking the next definition of the constant with respect to initial data: 
\begin{align}\label{est38}
I_2[u_0,u_1]\doteq\frac{1}{2}\left(I_0[u_0,u_1]+I_1[u_0,u_1]+\|u_1\|_{L^2}^2+\|\nabla u_0\|^2_{L^2}+\|\nabla u_1\|_{L^2}^2+\|\Delta u_0\|_{L^2}^2\right)
\end{align}
and applying \eqref{est5} and \eqref{est6}, it follows from \eqref{est4} that
\begin{align*}
(1+t)E[u](t)\leqslant I_2[u_0,u_1].
\end{align*}
Thus, the proof is complete.
%
\end{proof}
 
 To end this section, we have to mention that the estimate from \cite{Ikehata2001} holds
 \begin{align}\label{Solution.Itself}
 \|u(t,\cdot)\|_{L^2}^2\leqslant C I_2[u_0,u_1]
 \end{align}
 if the assumptions mentioned Lemma \ref{Lemma.Linear} are satisfied.

\section{Weighted estimates for nonlinear strongly damped wave equation}\label{Sec.Weighted.Nonlinear}
\setcounter{equation}{0}
Throughout this section, our motivation is to derive some exponentially weighted estimates for the next nonlinear problem:
\begin{equation}\label{nonhomo}
\begin{cases}
u_{tt}-\Delta u -\Delta u_t =F(u), &x\in \Omega,\,t>0,\\
u(0,x)=  u_0(x),\,u_t(0,x)=  u_1(x),&x\in \Omega,\\
u=0,&x\in \partial\Omega,\,t>0.
\end{cases}
\end{equation}
We assume that $\ee^{\psi(t,\cdot)}F(u)(t,\cdot)\in L^2(\Omega)$, where the weight function $\psi$ was introduced \eqref{Weighted.Fun}.

In order to derive exponentially weighted estimates in the $L^2$ norm, we introduce some lemmas, which plays an important role in the future.

\begin{lem} The weighted function $\psi$ fulfills two inequalities as follows:
	\begin{align}\label{est7}
	\frac{\varepsilon(2+\rho)+6}{\varepsilon(2+\rho)-2}|\nabla\psi(t,x)|^2-|\psi_t(t,x)|^2\leqslant0
	\end{align} 
	for any $t>0$ and $\rho>\rho_0$, and
	\begin{align}\label{est13}
	\frac{|\nabla\psi(t,x)|^2}{-\psi_t(t,x)}\leqslant \frac{2}{2+\rho}
	\end{align} 
	for any $t>0$ and $\rho> 0$.
\end{lem}
\begin{proof}
	We now begin to prove \eqref{est7}. By simple calculations, we get
	\begin{align*}
		&\frac{\varepsilon(2+\rho)+6}{\varepsilon(2+\rho)-2}|\nabla\psi(t,x)|^2-|\psi_t(t,x)|^2\\
		&=\left(\sqrt{\frac{\varepsilon(2+\rho)+6}{\varepsilon(2+\rho)-2}}\frac{|x|}{(1+t)^{2+\rho}}\right)^2-\left(\frac{2+\rho}{2}\frac{|x|^2}{(1+t)^{3+\rho}}+\frac{1}{(1+t)^{1+\rho}}\right)^2\\
		&=\frac{1}{(1+t)^{2+2\rho}}\left(\sqrt{\frac{\varepsilon(2+\rho)+6}{\varepsilon(2+\rho)-2}}X-\frac{2+\rho}{2}X^2-1\right)\left(\sqrt{\frac{\varepsilon(2+\rho)+6}{\varepsilon(2+\rho)-2}}X+\frac{2+\rho}{2}X^2+1\right)\leqslant 0,
	\end{align*}
	where $X=|x|/(1+t)$, because the discriminant of the first polynomial is 
	\begin{align*}
	\triangle =\frac{\varepsilon(2+\rho)+6}{\varepsilon(2+\rho)-2}-2(2+\rho)\leqslant 0.
	\end{align*}
	In the above, we have used the condition \eqref{epsilon} that
	\begin{align*}
	\frac{2}{2+\rho}<\left(2+\frac{8}{2\rho+3}\right)\frac{1}{2+\rho}=\frac{4\rho+14}{(2+\rho)(2\rho+3)}\leqslant \varepsilon.
	\end{align*}
	 Next, the desired estimate \eqref{est13} can be proved by the following way:
	\begin{align*}
		-\psi_t(t,x)=\frac{2+\rho}{2}\frac{|x|^2}{(1+t)^{3+\rho}}+\frac{1}{(1+t)^{1+\rho}}\geqslant\frac{2+\rho}{2}\frac{|x|^2}{(1+t)^{4+2\rho}}=\frac{2+\rho}{2}|\nabla\psi(t,x)|^2.
	\end{align*}
	So, the proof is completed
\end{proof}

\begin{lem} Let $u$ be a regular solution of \eqref{nonhomo}. Then, under the condition (\ref{epsilon}), the following estimate holds:
	\begin{align}\label{est9}
	\frac{|\nabla\psi(t,x)|^2}{-\psi_t(t,x)}\ee^{2\psi(t,x)}|u_{tt}(t,x)|^2&\leqslant -\psi_t(t,x)\ee^{2\psi(t,x)}|\Delta u(t,x)|^2+\varepsilon \,\ee^{2\psi(t,x)}|\Delta u_t(t,x)|^2\notag\\
	&\quad+C_{\varepsilon,\rho}\ee^{2\psi(t,x)}|F(u)(t,x)|^2,
	\end{align} 
	for all $t>0$ and $\rho>\rho_0$. Here, the positive constant $C_{\varepsilon,\rho}$ will be determined in \eqref{CoefficientCvr}.
\end{lem}
\begin{proof}
Taking the consideration on \eqref{nonhomo}, we may expand the quadratic term by
\begin{align*}
|u_{tt}(t,x)|^2&=\left| \Delta u(t,x)+ \Delta u_t(t,x)+F(u)(t,x)\right|^2\\
&=\left| \Delta u(t,x)+\Delta u_t(t,x)\right|^2+|F(u)(t,x)|^2+2\Delta u(t,x)F(u)(t,x)+2\Delta u_t(t,x)F(u)(t,x).
\end{align*}
Later, we will employ Young's inequality such that
\begin{align}\label{YoungIneq}
ab\leqslant \frac{1}{4\varepsilon_1} a^2+\varepsilon_1b^2
\end{align}
with $\varepsilon_1\doteq(\varepsilon(2+\rho)-2)/8>0$ (here we used the assumption \eqref{epsilon}).\\
Plugging $a=| \Delta u(t,x)|$ and $b=| \Delta u_t(t,x)|$ into the above Young's inequality, one immediately has
\begin{align*}
\left| \Delta u(t,x)+ \Delta u_t(t,x)\right|^2&=| \Delta u(t,x)|^2+ |\Delta u_t(t,x)|^2+2\Delta u(t,x)\Delta u_t(t,x)\\
&\leqslant\left(1+\frac{1}{2\varepsilon_1}\right)|\Delta u(t,x)|^2+(1+2\varepsilon_1)|\Delta u_t(t,x)|^2,
\end{align*}
which implies that
\begin{align*}
|u_{tt}(t,x)|^2&\leqslant \left(1+\frac{1}{2\varepsilon_1}\right)|\Delta u(t,x)|^2+(1+2\varepsilon_1)|\Delta u_t(t,x)|^2+|F(u)(t,x)|^2\\
&\quad+2\Delta u(t,x)F(u)(t,x)+2\Delta u_t(t,x)F(u)(t,x).
\end{align*}
On the other hand, by employing Young's inequality \eqref{YoungIneq} again, we can compute
\begin{align*}
&2\Delta u(t,x)F(u)(t,x)+2\Delta u_t(t,x)F(u)(t,x)\\
&\leqslant \frac{1}{2\varepsilon_1}|\Delta u(t,x)|^2+2\varepsilon_1|\Delta u_t(t,x)|^2+\left(2\varepsilon_1+\frac{1}{2\varepsilon_1}\right)|F(u)(t,x)|^2.
\end{align*}
Hence,
\begin{align*}
|u_{tt}(t,x)|^2\leqslant \left(1+\frac{1}{\varepsilon_1}\right)|\Delta u(t,x)|^2+\left(1+4\varepsilon_1\right)|\Delta u_t(t,x)|^2+ \left(1+2\varepsilon_1+\frac{1}{2\varepsilon_1}\right)|F(u)(t,x)|^2.
\end{align*}

Finally, we may deduce that
\begin{align*}
	\frac{|\nabla\psi(t,x)|^2}{-\psi_t(t,x)}\ee^{2\psi(t,x)}|u_{tt}(t,x)|^2&\leqslant\left(\frac{\varepsilon(2+\rho)+6}{\varepsilon(2+\rho)-2}\right)\frac{|\nabla\psi(t,x)|^2}{-\psi_t(t,x)}\ee^{2\psi(t,x)}|\Delta u(t,x)|^2\\
	&\quad+\varepsilon\left(1+\frac{\rho}{2}\right)\frac{|\nabla\psi(t,x)|^2}{-\psi_t(t,x)}\ee^{2\psi(t,x)}|\Delta u_t(t,x)|^2+C_{\varepsilon,\rho}\ee^{2\psi(t,x)}|F(u)(t,x)|^2\\
	&\leqslant -\psi_t(t,x)\ee^{2\psi(t,x)}|\Delta u(t,x)|^2+\varepsilon\,\ee^{2\psi(t,x)}|\Delta u_t(t,x)|^2\\
	&\quad+C_{\varepsilon,\rho}e^{2\psi(t,x)}|F(u)(t,x)|^2,
\end{align*}
where we have used \eqref{est7} as well as \eqref{est13}, and the fact that $1+4\varepsilon_1=\varepsilon(1+\rho/2)$. In the above inequality, we denote the positive constant
\begin{align}\label{CoefficientCvr}
C_{\varepsilon,\rho}\doteq\frac{2}{2+\rho}\left(1+\frac{\varepsilon(2+\rho)-2}{4}+\frac{4}{\varepsilon(2+\rho)-2}\right)>0.
\end{align}
Thus, we complete the proof of this lemma.
\end{proof}

By using our derived lemmas, we may prove the next propositions for weighted estimates for higher-order energy to \eqref{nonhomo}. At this time, we restrict ourselves $\rho>\rho_0$ to control the higher-order term in the energy estimate.
\begin{prop}\label{prop1} Let $u$ be a regular solution of \eqref{nonhomo}. Then, we have the estimate
	\begin{align}\label{est10}
	&\left\|\ee^{\psi(t,\cdot)}\nabla u_t(t,\cdot)\right\|_{L^2}^2+\left\|\ee^{\psi(t,\cdot)}\Delta u(t,\cdot)\right\|_{L^2}^2+ (1-\varepsilon)\int_0^t\left\|\ee^{\psi(s,\cdot)}\Delta u_t(s,\cdot)\right\|_{L^2}^2\dd s\notag\\
	&\leqslant \left\|\ee^{\psi(0,\cdot)}\nabla u_1\right\|_{L^2}^2+\left\|\ee^{\psi(0,\cdot)}\Delta u_0\right\|_{L^2}^2+\widetilde{C}_{\varepsilon,\rho}\int_0^t\left\|\ee^{\psi(s,\cdot)}F(u)(s,\cdot)\right\|^2_{L^2}\dd s
	\end{align} 
	for all $t>0$ and $\rho>\rho_0$. Here, the positive constant $\widetilde{C}_{\varepsilon,\rho}$ will be shown in \eqref{EQQQ02}.
\end{prop}
\begin{proof}
Firstly, we multiply \eqref{nonhomo} by $e^{2\psi(t,x)}\Delta u_{t}$ to have
\begin{align}\label{est11}
&\ee^{2\psi(t,x)}u_{tt}(t,x)\Delta u_t(t,x)-\ee^{2\psi(t,x)}\Delta u(t,x) \Delta u_t(t,x)-\ee^{2\psi(t,x)}|\Delta u_t(t,x)|^2\notag\\
&=\ee^{2\psi(t,x)}F(u)(t,x)\Delta u_t(t,x).
\end{align} 
Due to the computations that
\begin{align*}
	\ee^{2\psi(t,x)}u_{tt}(t,x)\Delta u_t(t,x)&=\ee^{2\psi(x,t)}\divv(u_{tt}(t,x)\nabla u_t(t,x))- \ee^{2\psi(t,x)}(\nabla u_t(t,x))_t\cdot\nabla u_t(t,x)\\
	&=\divv\left(\ee^{2\psi(t,x)}u_{tt}(t,x)\nabla u_t(t,x)\right)-2e^{2\psi(t,x)}u_{tt}(t,x)\nabla\psi(t,x)\cdot\nabla u_t(t,x)\\
	&\quad- \frac{1}{2}\frac{\dd}{\dd t}\left(\ee^{2\psi(t,x)}|\nabla u_t(t,x)|^2\right)+\ee^{2\psi(t,x)}\psi_t(t,x)|\nabla u_t(t,x)|^2,
\end{align*}
and
\begin{align*}
\ee^{2\psi(t,x)}\Delta u(t,x)\Delta u_t(t,x)=\frac{1}{2}\frac{\dd}{\dd t}\left(\ee^{2\psi(t,x)}|\Delta u(t,x)|^2\right)- \ee^{2\psi(t,x)}\psi_t(t,x)|\Delta u(t,x)|^2,
\end{align*}
it follows from \eqref{est11} that
\begin{align}\label{est12}
&\frac{\dd}{\dd t}\left(\frac{\ee^{2\psi(t,x)}}{2}\left(|\nabla u_t(t,x)|^2+|\Delta u(t,x)|^2\right)\right)-\divv\left(\ee^{2\psi(t,x)}u_{tt}(t,x)\nabla u_t(t,x)\right)\nonumber\\
&+2\ee^{2\psi(t,x)}u_{tt}(t,x)\nabla\psi(t,x)\cdot\nabla u_t(t,x)-\ee^{2\psi(t,x)}\psi_t(t,x)|\nabla u_t(t,x)|^2\nonumber\\
&+\ee^{2\psi(t,x)}(-\psi_t(t,x))|\Delta u(t,x)|^2+\ee^{2\psi(t,x)}|\Delta u_t(t,x)|^2\notag\\
&=-\ee^{2\psi(t,x)}F(u)(t,x)\Delta u_t(t,x).
\end{align}
Clearly,
\begin{align*}
&2e^{2\psi(t,x)}u_{tt}(t,x)\nabla\psi(t,x)\cdot\nabla u_t(t,x)-\ee^{2\psi(t,x)}\psi_t(t,x)|\nabla u_t(t,x)|^2\\
&=\frac{\ee^{2\psi(t,x)}}{-\psi_t(t,x)}\left|\psi_t(t,x)\nabla u_t(t,x)-\nabla\psi(t,x) u_{tt}(t,x)\right|^2+\frac{|\nabla\psi(t,x)|^2}{\psi_t(t,x)}\ee^{2\psi(t,x)}|u_{tt}(t,x)|^2.
\end{align*}
It leads that  \eqref{est12} can be written by
\begin{align}\label{EQQQ01}
	&\frac{\dd}{\dd t}\left(\frac{\ee^{2\psi(t,x)}}{2}\left(|\nabla u_t(t,x)|^2+|\Delta u(t,x)|^2\right)\right)-\divv\left(\ee^{2\psi(t,x)}u_{tt}(t,x)\nabla u_t(t,x)\right)\notag\\
	&+\frac{\ee^{2\psi(t,x)}}{-\psi_t(t,x)}\left|\psi_t(t,x)\nabla u_t(t,x)-\nabla\psi(t,x) u_{tt}(t,x)\right|^2\notag\\
	&+\ee^{2\psi(t,x)}(-\psi_t(t,x))|\Delta u(t,x)|^2+\ee^{2\psi(t,x)}|\Delta u_t(t,x)|^2\notag\\
	&\leqslant\frac{|\nabla\psi(t,x)|^2}{-\psi_t(t,x)}\ee^{2\psi(t,x)}|u_{tt}(t,x)|^2+\ee^{2\psi(t,x)}|F(u)(t,x)||\Delta u_t(t,x)|\notag\\
	&\leqslant \ee^{2\psi(t,x)}(-\psi_t(t,x))|\Delta u(t,x)|^2+\varepsilon \,\ee^{2\psi(t,x)}|\Delta u_t(t,x)|^2\notag\\
	&\quad+C_{\varepsilon,\rho}\ee^{2\psi(t,x)}|F(u)(t,x)|^2+\ee^{2\psi(t,x)}|F(u)(t,x)||\Delta u_t(t,x)|,
\end{align}
where we have used \eqref{est9} in the last step of the above estimate. Then, the above inequality \eqref{EQQQ01} can be simplified as follows:
\begin{align*}
	&\frac{\dd}{\dd t}\left(\frac{\ee^{2\psi(t,x)}}{2}\left(|\nabla u_t(t,x)|^2+|\Delta u(t,x)|^2\right)\right)-\divv\left(\ee^{2\psi(t,x)}u_{tt}(t,x)\nabla u_t(t,x)\right)\\
	&+\frac{\ee^{2\psi(t,x)}}{-\psi_t(t,x)}\left|\psi_t(t,x)\nabla u_t(t,x)-\nabla\psi(t,x) u_{tt}(t,x)\right|^2+(1-\varepsilon)\ee^{2\psi(t,x)}|\Delta u_t(t,x)|^2\\
	&\leqslant C_{\varepsilon,\rho}\ee^{2\psi(t,x)}|F(u)(t,x)|^2+\ee^{2\psi(t,x)}|F(u)(t,x)||\Delta u_t(t,x)|.
\end{align*}
The application of Young's inequality yields
\begin{align*}
\ee^{2\psi(t,x)}|F(u)(t,x)||\Delta u_t(t,x)|\leqslant \frac{1-\varepsilon}{2}e^{2\psi(t,x)}|\Delta u_t(t,x)|^2+ \frac{1}{2(1-\varepsilon)}\ee^{2\psi(t,x)}|F(u)(t,x)|^2,
\end{align*}
since our setting of the constant that $\varepsilon<1$. This inequality shows that
\begin{align}\label{EQQQ02}
	&\frac{\dd}{\dd t}\left(\frac{\ee^{2\psi(t,x)}}{2}\left(|\nabla u_t(t,x)|^2+|\Delta u(t,x)|^2\right)\right)-\divv\left(\ee^{2\psi(t,x)}u_{tt}(t,x)\nabla u_t(t,x)\right)\notag\\
&+\frac{\ee^{2\psi(t,x)}}{-\psi_t(t,x)}\left|\psi_t(t,x)\nabla u_t(t,x)-\nabla\psi(t,x) u_{tt}(t,x)\right|^2+\frac{1-\varepsilon}{2}\ee^{2\psi(t,x)}|\Delta u_t(t,x)|^2\notag\\
	&\leqslant\left(C_{\varepsilon,\rho}+\frac{1}{2(1-\varepsilon)}\right) \ee^{2\psi(t,x)}|F(u)(t,x)|^2\doteq\widetilde{C}_{\varepsilon,\rho}\ee^{2\psi(t,x)}|F(u)(t,x)|^2.
\end{align}
Consequently, integrating the above inequality over $\Omega\times[0,t]$ and using the boundary condition and the fact that $\psi_t<0$ we may derive our desired result.
\end{proof}

Next, the lower-order weighted energy can be estimated by the next lemma. We should emphasize that the next lemma is hold for all $\rho>0$.

\begin{prop}\label{prop2} Let $u$ be a regular solution of \eqref{nonhomo}. Then, we have the estimate
	\begin{align}\label{est14}
	&\left\|\ee^{\psi(t,\cdot)}u_t(t,\cdot)\right\|_{L^2}^2+\left\|\ee^{\psi(t,\cdot)}\nabla u(t,\cdot)\right\|_{L^2}^2\notag\\
	&\leqslant \left\|\ee^{\psi(0,\cdot)} u_1\right\|_{L^2}^2+\left\|\ee^{\psi(0,\cdot)}\nabla u_0\right\|_{L^2}^2+2\int_0^t\left\|\ee^{\psi(s,\cdot)}F(u)(s,\cdot)\right\|_{L^2}\left\|\ee^{\psi(s,\cdot)}u_t(s,\cdot)\right\|_{L^2}\dd s
	\end{align} 
	for all $t>0$ and $\rho>0$.
\end{prop}
\begin{proof}
Similarly to the proof of Lemma 2.1 in \cite{Ikehata-Inoue2008}, we multiply \eqref{nonhomo} by $e^{2\psi(t,x)}u_{t}$ and get
\begin{align*}
\ee^{2\psi(t,x)}u_{tt}(t,x) u_t(t,x)-\ee^{2\psi(t,x)}\Delta u(t,x) u_t(t,x)-\ee^{2\psi(t,x)}\Delta u_t(t,x) u_t(t,x)=e^{2\psi(t,x)} F(u)(t,x) u_t(t,x).
\end{align*}
Hence, the equality can be deduced as follows:
\begin{align*}
	&\frac{\dd}{\dd t}\left(\frac{e^{2\psi(t,x)}}{2}\left(|u_t(t,x)|^2+|\nabla u(t,x)|^2\right)\right)-\divv\left(\ee^{2\psi(t,x)}u_{t}(t,x)\nabla(u(t,x)+ u_t(t,x))\right)\\
	&+\frac{\ee^{2\psi(t,x)}}{\psi_t(t,x)}|u_t(t,x)|^2\left(|\nabla \psi(t,x)|^2-\psi_t(t,x)|\nabla \psi(t,x)|^2-|\psi_t(t,x)|^2\right)\\
	&+\frac{\ee^{2\psi(t,x)}}{-\psi_t(t,x)}\left|\psi_t(t,x)\nabla u(t,x)-u_t(t,x)\nabla\psi(t,x)\right|^2+\ee^{2\psi(t,x)}\left|\nabla u_t(t,x)+u_t(t,x)\nabla\psi(t,x)\right|^2\\
	&=\ee^{2\psi(t,x)}F(u)(t,x) u_t(t,x).
\end{align*}
Using \eqref{est15} and the property $\psi_t<0$, we claim that
\begin{align}\label{EQQQ03}
&\frac{\dd}{\dd t}\left(\frac{\ee^{2\psi(t,x)}}{2}\left(|u_t(t,x)|^2+|\nabla u(t,x)|^2\right)\right)-\divv\left(\ee^{2\psi(t,x)}u_{t}(t,x)\nabla(u(t,x)+ u_t(t,x))\right)\notag\\
&\leqslant \ee^{2\psi(t,x)}F(u)(t,x) u_t(t,x).
\end{align}
Let us integrate \eqref{EQQQ03} over $\Omega\times[0,t]$ to obtain
\begin{align*}
	&\frac{1}{2}\left\|\ee^{\psi(t,\cdot)}u_t(t,\cdot)\right\|_{L^2}^2+\frac{1}{2}\left\|\ee^{\psi(t,\cdot)}\nabla u(t,\cdot)\right\|_{L^2}^2\\
	&\leqslant \frac{1}{2}\left\|\ee^{\psi(0,\cdot)} u_1\right\|_{L^2}^2+\frac{1}{2}\left\|\ee^{\psi(0,\cdot)}\nabla u_0\right\|_{L^2}^2+\int_0^t\left\|\ee^{2\psi(s,\cdot)}F(u)(s,\cdot)u_t(s,\cdot)\right\|_{L^1}\dd s.
\end{align*} 
The proof can be completed after using H\"older's inequality in the above inequality.
\end{proof}

\section{Proof of Theorem \ref{Thm.GESDS}}\label{Sec.Proof.Thm}
\setcounter{equation}{0}
Before proving our main theorem, let us denote by $E_0(t,x)$ and $E_1(t,x)$ the fundamental solutions to the linear problem \eqref{eq1} with initial data $(u_0,u_1)=(\delta_0,0)$ and $(u_0,u_1)=(0,\delta_0)$, respectively. Here, $\delta_0$ is the Dirac distribution in $x=0$ with respect to spatial variables. Therefore, the solution $u^{\lin}=u^{\lin}(t,x)$ to the exterior problem \eqref{eq1} is given by
\begin{align*}
u^{\lin}(t,x)=E_0(t,x)\ast_{(x)}u_0(x)+E_1(t,x)\ast_{(x)}u_1(x).
\end{align*}

Let us define an evolution space
\begin{align*}
X(T)\doteq\ml{C}\left([0,T],H^2(\Omega)\cap H^1_0(\Omega)\right)\cap\ml{C}^1\left([0,T],H^1(\Omega)\right)
\end{align*}
carrying the corresponding norm
\begin{align*}
M[u](T)\doteq\sup\limits_{t\in[0,T]}W[u](t)\doteq\sup\limits_{t\in[0,T]}\left(\left\|\ee^{\psi(t,\cdot)}\ml{D}u(t,\cdot)\right\|_{L^2}^2+(1+t)\left\|\ml{D}u(t,\cdot)\right\|^2_{L^2}+\left\|u(t,\cdot)\right\|^2_{L^2}\right),
\end{align*}
where the space-time differential operator is denoted by $\ml{D}\doteq\left(\partial_t,\nabla,\nabla\partial_t,\Delta\right)$.

According to Duhamel's principle, we introduce the operator
\begin{align*}
N:\, u\in X(T)\longrightarrow Nu&\doteq u^{\lin}(t,x)+u^{\non}(t,x)\\
&\doteq u^{\lin}(t,x)+\int_0^tE_1(t-s,x)\ast_{(x)}F(u)(s,x)\,\dd s,
\end{align*}
where we choose $F(u)(t,x)\doteq|u(t,x)|^p+|u_t(t,x)|^q$ in this section. Furthermore, let us define 
\begin{align*}
J_0[u_0,u_1]\doteq I_2[u_0,u_1]+I_{\mathrm{exp}}[u_0,u_1],
\end{align*}
where $I_{\mathrm{exp}}[u_0,u_1]$ has been defined in \eqref{Weighted.Data}. We should remark that if $\ml{J}[u_0,u_1]\leqslant\varepsilon_0$, then it is trivial that $J_0[u_0,u_1]\leqslant C\varepsilon_0$ for some constant $C>0$.

We will prove as the global in time solution to \eqref{Eq.Semi.BVP} the fixed points of operator $N$. In other words, our first aim is to derive
\begin{align}
M[Nu](T)\leqslant \widetilde{C}_0J_0[u_0,u_1]+\widetilde{C}_1\left(\sum\limits_{r=p,q,(p+1)/2,(q+1)/2}M[u](T)^r+\widetilde{M}[u](T;p,q)\right)\label{Important.01}
\end{align}
with positive constants $\widetilde{C}_0$ and $\widetilde{C}_1$.
Here, we denote
\begin{align*}
\widetilde{M}[u](t;p,q)\doteq I_{\mathrm{exp}}[u_0,u_1]^{\frac{q-1}{q}}+M[u](t)^{\frac{p(q-1)}{q}}+M[u](t)^{q-1}.
\end{align*}
Moreover, to guarantee uniqueness of global (in time) small data solution, our second aim is to prove the next estimates:
\begin{align}\label{Important.02}
M[Nu-Nv]\leqslant \widetilde{C}_2 M[u-v],
\end{align}
for any $u,v\in X(T)$, with a positive constant $\widetilde{C}_2$.

From Lemma \ref{Lemma.Linear} and \eqref{Solution.Itself}, it is sufficient for us to show
\begin{align*}
\|u^{\lin}(t,\cdot)\|_{L^2}^2+(1+t)\|\ml{D}u^{\lin}(t,\cdot)\|_{L^2}^2\leqslant CI_2[u_0,u_1].
\end{align*}
Furthermore, the association of Lemmas \ref{prop1} and \ref{prop2} shows that
\begin{align*}
\left\|\ee^{\psi(t,\cdot)}\ml{D}u^{\lin}(t,\cdot)\right\|_{L^2}^2\leqslant C I_{\mathrm{exp}}[u_0,u_1].
\end{align*}
Together with them, we claim that $u^{\lin}\in X(T)$. 

Therefore, the next part of this section is to prove $u^{\non}\in X(T)$. To estimate the nonlinear term in the weighted space, we introduce some lemmas, which will be used later.
\color{black}
\begin{lem}\label{lemma1} Let $p,q>4+\rho$ for all $\rho>0$. Then, the next estimate holds:
	\begin{align}\label{est16}
	\int_0^t\left\|\ee^{\psi(s,\cdot)}F(u)(s,\cdot)\right\|_{L^2}^2\dd s \leqslant C\left(M[u](t)^p+M[u](t)^q\right)
	\end{align}
	for all $t>0$.
\end{lem}
\begin{proof} To begin with the proof, let us split the estimate into two parts
\begin{align*}
	\int_0^t\left\|\ee^{\psi(s,\cdot)}F(u)(s,\cdot)\right\|_{L^2}^2\dd s&\leqslant 2\int_0^t\left\|\ee^{\psi(s,\cdot)}|u(s,\cdot)|^p\right\|_{L^2}^2\dd s+2\int_0^t\left\|\ee^{\psi(s,\cdot)}|u_t(s,\cdot)|^q\right\|_{L^2}^2\dd s\\
	&\doteq 2\left(A_1(t)+A_2(t)\right).
\end{align*}
Now, we will estimate each part step by step. By using Lemma \ref{Lem.G-N}, i.e. the Gagliardo-Nirenberg type inequality associated with weighted function $\psi$, we get 
\begin{align*}
	A_1(t)&\leqslant C\int_0^t\left((1+s)^{\frac{(2+\rho)(1-\theta(2p))}{2}}\|\nabla u(s,\cdot)\|_{L^2}^{1-\frac{1}{p}}\left\|\ee^{\psi(s,\cdot)}\nabla u(s,\cdot)\right\|_{L^2}^{\frac{1}{p}}\right)^{2p}\dd s\\
	&\leqslant C\int_0^t(1+s)^{-(1+\eta)}\left((1+s)^{(2+\rho)(1-\theta(2p))+\frac{1+\eta}{p}-(1-\frac{1}{p})}W[u](s)\right)^{p}\dd s\\
	&\leqslant C\int_0^t(1+s)^{-(1+\eta)}\,\dd s \left(\sup_{s\in[0,t]}(1+s)^{\beta_1}W[u](s)\right)^{p}\\
	&\leqslant C \left(\sup_{s\in[0,t]}(1+s)^{\beta_1}W[u](s)\right)^{p},
\end{align*}
where 
\begin{align*}
\beta_1\doteq(2+\rho)(1-\theta(2p))+\frac{1+\eta}{p}-\left(1-\frac{1}{p}\right)
\end{align*}
and $\theta(2p)=2(\frac{1}{2}-\frac{1}{2p})$ for all $\eta>0$. Choosing a sufficiently small constant $\eta>0$, we found that
\begin{align*}
\beta_1=(2+\rho)\left(1-2\left(\frac{1}{2}-\frac{1}{2p}\right)\right)+\frac{2}{p}+\frac{\eta}{p}-1=\frac{\eta}{p}-\frac{p-4-\rho}{p}<0,
\end{align*}
since our assumption $p>4+\rho$. Therefore, the estimate for $A_1(t)$ is
\begin{align*}
A_1(t)\leqslant  CM[u](t)^p.
\end{align*}

Similarly as the above, one has
\begin{align*}
	A_2(t)&\leqslant C\int_0^t\left((1+s)^{\frac{(2+\rho)(1-\theta(2q))}{2}}\|\nabla u_t(s,\cdot)\|_{L^2}^{1-\frac{1}{q}}\left\|\ee^{\psi(s,\cdot)}\nabla u_t(s,\cdot)\right\|_{L^2}^{\frac{1}{q}}\right)^{2q}\dd s\\
	&\leqslant  C\int_0^t(1+s)^{-(1+\eta)}\left((1+s)^{(2+\rho)(1-\theta(2q))+\frac{1+\eta}{q}-(1-\frac{1}{q})}W[u](s)\right)^q\dd s\\
	&\leqslant C\int_0^t(1+s)^{-(1+\eta)}\dd s \left(\sup_{s\in[0,t]}(1+s)^{\beta_2}W[u](s)\right)^q\\
	&\leqslant C \left(\sup_{s\in[0,t]}(1+s)^{\beta_2}W[u](s)\right)^q,
\end{align*}
where 
\begin{align*}
\beta_2\doteq(2+\rho)(1-\theta(2q))+\frac{1+\eta}{q}-\left(1-\frac{1}{q}\right)
\end{align*}
and $\theta(2q)=2(\frac{1}{2}-\frac{1}{2q})$ for all $\eta>0$. By choosing a sufficiently small constant $\eta>0$. It is clear that $\beta_2<0$, when $q>4+\rho$. Thus, the second time-dependent function can be estimated by
\begin{align*}
A_2(t)\leqslant  CM[u](t)^q.
\end{align*}
Summarizing the above estimates, we may complete the proof.
\end{proof}

Additionally, by the similar method as the proof of Lemma \ref{lemma1}, we will prove the following result.
\begin{lem}\label{lemma2}
	Let $p,q>5+\rho$ for all $\rho>0$. Then, the next estimate holds:
	\begin{equation}\label{est17}
	\int_0^t\left\|\ee^{\psi(s,\cdot)}F(u)(s,\cdot)\right\|_{L^2}\left\|\ee^{\psi(s,\cdot)}u_t(s,\cdot)\right\|_{L^2}\dd s \leqslant C\left(M[u](t)^{\frac{p+1}{2}}+M[u](t)^{\frac{q+1}{2}}\right)
	\end{equation}
	for all $t>0$.
\end{lem}
\begin{proof}
	Let us first divide the integral into two parts such that
	\begin{align*}
		&\int_0^t\left\|\ee^{\psi(s,\cdot)}F(u)(s,\cdot)\right\|_{L^2}\left\|\ee^{\psi(s,\cdot)}u_t(s,\cdot)\right\|_{L^2}\dd s\\
		&\leqslant 2\int_0^t\left\|\ee^{\psi(s,\cdot)}|u(s,\cdot)|^p\right\|_{L^2}\left\|\ee^{\psi(s,\cdot)}u_t(s,\cdot)\right\|_{L^2}\dd s+2\int_0^t\left\|\ee^{\psi(s,\cdot)}|u_t(s,\cdot)|^q\right\|_{L^2}\left\|\ee^{\psi(s,\cdot)}u_t(s,\cdot)\right\|_{L^2}\dd s\\
		&\doteq 2\left(B_1(t)+B_2(t)\right).
	\end{align*}
	We apply Lemma \ref{Lem.G-N} again to derive
	\begin{align*}
		B_1(t)&\leqslant C\int_0^t\left((1+s)^{\frac{(2+\rho)(1-\theta(2p))}{2}}\left\|\nabla u(s,\cdot)\right\|_{L^2}^{1-\frac{1}{p}}\left\|\ee^{\psi(s,\cdot)}\nabla u(s,\cdot)\right\|^{\frac{1}{p}}_{L^2}\right)^{p}W[u](s)^{\frac{1}{2}}\,\dd s\\
		&\leqslant C\int_0^t(1+s)^{-(1+\eta)}\left((1+s)^{\frac{(2+\rho)(1-\theta(2p))}{2}+\frac{1+\eta}{p}-\frac{1}{2}(1-\frac{1}{p})}W[u](s)^{\frac{1}{2}}\right)^{p}W[u](s)^{\frac{1}{2}}\,\dd s\\
		&\leqslant C\int_0^t(1+s)^{-(1+\eta)}\,\dd s \left(\sup_{s\in[0,t]}(1+s)^{\beta_3}W[u](s)^{\frac{1}{2}}\right)^{p}M[u](t)^{\frac{1}{2}}\\
		&\leqslant C \left(\sup_{s\in[0,t]}(1+s)^{\beta_3}W[u](s)^{\frac{1}{2}}\right)^{p}M[u](t)^{\frac{1}{2}}
	\end{align*}
	for all $\eta>0$. Here, choosing sufficiently small constant $\eta>0$, we observe that
	\begin{align*}
	\beta_3\doteq\frac{(2+\rho)(1-\theta(2p))}{2}+\frac{1+\eta}{p}-\frac{1}{2}\left(1-\frac{1}{p}\right)=\frac{\eta}{p}-\frac{p-5-\rho}{2p}<0,
	\end{align*}
	where we used our assumption $p>5+\rho$. So, the estimate holds
	\begin{align*}
	B_1(t)\leqslant  CM[u](t)^{\frac{p+1}{2}}.
	\end{align*}

	 Analogously, we may compute
	\begin{align*}
		B_2(t)&\leqslant C\int_0^t\left((1+s)^{\frac{(2+\rho)(1-\theta(2q))}{2}}\|\nabla u_t(s,\cdot)\|_{L^2}^{1-\frac{1}{q}}\left\|\ee^{\psi(s,\cdot)}\nabla u_t(s,\cdot)\right\|_{L^2}^{\frac{1}{q}}\right)^{q}W[u](s)^{\frac{1}{2}}\,\dd s\\
		&\leqslant C\int_0^t(1+s)^{-(1+\eta)}\,\dd s \left(\sup_{s\in[0,t]}(1+s)^{\beta_4}W[u](s)^{\frac{1}{2}}\right)^{q}M[u](t)^{\frac{1}{2}}\\
		&\leqslant C \left(\sup_{s\in[0,t]}(1+s)^{\beta_4}W[u](s)^{\frac{1}{2}}\right)^{q}M[u](t)^{\frac{1}{2}},
	\end{align*}
	for all $\eta>0$. The parameter $\beta_4$ satisfies
	\begin{align*}
	\beta_4\doteq\frac{(2+\rho)(1-\theta(2q))}{2}+\frac{1+\eta}{q}-\frac{1}{2}\left(1-\frac{1}{q}\right)=\frac{\eta}{q}-\frac{q-5-\rho}{2q}<0
	\end{align*}
	by using our assumption $q>5+\rho$ and the choice of sufficiently small constant $\eta>0$. We conclude
	\begin{align*}
	B_2(t)\leqslant  CM[u](t)^{\frac{q+1}{2}}.
	\end{align*}
	This implies the desired estimate.
\end{proof}
The next proposition plays an important role in proving global (in time) existence of small data solution by using Lemmas \ref{lemma1} and \ref{lemma2} directly.
\begin{prop}\label{prop3}
	Let $p,q>5+\rho$ for all $\rho>\rho_0$. The next estimate holds:
	\begin{equation}\label{est18}
	\left\|\ee^{\psi(t,\cdot)}\ml{D}u^{\non}(t,\cdot)\right\|_{L^2}^2 \leqslant C\left(M[u](t)^p+M[u](t)^q+M[u](t)^{\frac{p+1}{2}}+M[u](t)^{\frac{q+1}{2}}\right)
	\end{equation}
	for all $t>0$.
\end{prop}

In conclusion, we has completed the estimate of weighted energy even for higher-order. So, we will estimate
\begin{align*}
\|\ml{D}u^{\non}(t,\cdot)\|^2_{L^2}\quad\mbox{and}\quad \|u^{\non}(t,\cdot)\|^2_{L^2}
\end{align*}
in the following part.

Let us define integral operators
\begin{align*}
u_{N_1}(t,x)\doteq\int_0^tE_1(t-s,x)\ast_{(x)}|u(s,x)|^p\,\dd s\quad\mbox{and}\quad u_{N_2}(t,x)\doteq\int_0^tE_1(t-s,x)\ast_{(x)}|u_t(s,x)|^q\,\dd s,
\end{align*}
which implies
\begin{align*}
u^{\non}(t,x)=u_{N_1}(t,x)+u_{N_2}(t,x).
\end{align*}
Moreover, we define the differential operators
\begin{align*}
\ml{D}_1\doteq(\partial_t,\nabla)\quad\mbox{and}\quad\ml{D}_2\doteq(\partial_t\nabla,\Delta).
\end{align*}
Hence, we notice that $\ml{D}=(\ml{D}_1,\ml{D}_2)$.

From the proof of Theorem 1.1 in \cite{Ikehata-Inoue2008}, under our assumption $p>6+2\rho_0$ it is sufficient to derive
\begin{align}\label{est25}
\|u_{N_1}(t,\cdot)\|_{L^2}^2+(1+t)\|\ml{D}_1u_{N_1}(t,\cdot)\|_{L^2}^2\leqslant CM[u](t)^{p}.
\end{align}

Next, we will estimate $\|\ml{D}_2u_{N_1}(t,\cdot)\|_{L^2}$. The application of \eqref{EQQQ06} indicates that
\begin{align*}
&\left\|\ml{D}_2E_1(t-s,\cdot)\ast_{(\cdot)}|u(s,\cdot)|^p\right\|_{L^2}\notag\\
&\leqslant C(1+t-s)^{-\frac{1}{2}}\left(\left\||u(s,\cdot)|^p\right\|_{L^2}+\left\|\nabla |u(s,\cdot)|^p\right\|_{L^2}\right)\nonumber\\
&\leqslant C(1+t-s)^{-\frac{1}{2}}\left(\|u(s,\cdot)\|^p_{L^{2p}}+\left\||u(s,\cdot)|^{p-1}\nabla u(s,\cdot)\right\|_{L^2}\right)\nonumber\\
&\leqslant C(1+t-s)^{-\frac{1}{2}}\left(\|u(s,\cdot)\|^p_{L^{2p}}+\left\||u(s,\cdot)|^{p-1}\right\|_{L^{2p/(p-1)}}\|\nabla u(s,\cdot)\|_{L^{2p}}\right)\nonumber\\
&=C(1+t-s)^{-\frac{1}{2}}\left(\|u(s,\cdot)\|^p_{L^{2p}}+\|u(s,\cdot)\|^{p-1}_{L^{2p}}\|\nabla u(s,\cdot)\|_{L^{2p}}\right),
\end{align*}
where we have used  H\"older's inequality.\\
Moreover, due to $u(t,\cdot)\in H^2(\Omega)$ for $t\in[0,T]$, the Gagliardo-Nirenberg inequality shows
\begin{align*}
\|u(s,\cdot)\|_{L^{2p}}&\leqslant C\|u(s,\cdot)\|_{L^2}^{1-\theta_1(2p)}\|\nabla u(s,\cdot)\|_{L^2}^{\theta_1(2p)},\\
\|\nabla u(s,\cdot)\|_{L^{2p}}&\leqslant C\|\nabla u(s,\cdot)\|_{L^2}^{1-\theta_1(2p)}\|\Delta u(s,\cdot)\|_{L^2}^{\theta_1(2p)},
\end{align*}
where $\theta_1(2p):=\theta(2p)=2(\frac{1}{2}-\frac{1}{2p})=1-\frac{1}{p}$.\\
By applying the estimates from the definition of solution space $X(T)$ such that
\begin{align}
\|u(s,\cdot)\|_{L^{2p}}&\leqslant C(1+s)^{-\frac{\theta_1(2p)}{2}}W[u](s)^{\frac{1}{2}},\label{est27}\\
\|\nabla u(s,\cdot)\|_{L^{2p}}&\leqslant  C(1+s)^{-\frac{1}{2}}W[u](s)^{\frac{1}{2}},\label{est28}
\end{align}
we derive
\begin{align*}
\left\|\ml{D}_2E_1(t-s,\cdot)\ast_{(\cdot)}|u(s,\cdot)|^p\right\|_{L^2}\leqslant C(1+t-s)^{-\frac{1}{2}}(1+s)^{-\frac{\theta_1(2p)p}{2}}W[u](s)^{\frac{p}{2}}.
\end{align*}
From the definition of $u_{N_1}$, we immediately obtain
\begin{align}\label{est29}
\|\ml{D}_2u_{N_1}(t,\cdot)\|_{L^2}&\leqslant\int_0^t\left\|\ml{D}_2E_1(t-s,\cdot)\ast_{(\cdot)}|u(s,\cdot)|^p\right\|_{L^2}\dd s\nonumber\\
&\leqslant C\int_0^t(1+t-s)^{-\frac{1}{2}}(1+s)^{-\frac{\theta_1(2p)p}{2}}W[u](s)^{\frac{p}{2}}\,\dd s\nonumber\\
&\leqslant CM[u](t)^{\frac{p}{2}}\int_0^t(1+t-s)^{-\frac{1}{2}}(1+s)^{-\frac{\theta_1(2p)p}{2}}\,\dd s\nonumber\\
&\leqslant C(1+t)^{-\frac{1}{2}}M[u](t)^{\frac{p}{2}},
\end{align}
where we used Lemma 4.1 in \cite{Cui2001} and the condition that $\theta_1(2p)p/2>1$ under the assumption $p>3$.

Now, we begin with the estimate of $\|\ml{D}_2u_{N_2}(t,\cdot)\|_{L^2}$. Taking the consideration of \eqref{EQQQ06} again with H\"older's inequality, we have
\begin{align}\label{est30}
&\left\|\ml{D}_2E_1(t-s,\cdot)\ast_{(\cdot)}|u_t(s,\cdot)|^q\right\|_{L^2}\notag\\
&\leqslant C(1+t-s)^{-\frac{1}{2}}\left(\left\||u_t(s,\cdot)|^q\right\|_{L^2}+\left\|\nabla |u_t(s,\cdot)|^q\right\|_{L^2}\right)\nonumber\\
&\leqslant C(1+t-s)^{-\frac{1}{2}}\left(\|u_t(s,\cdot)\|^q_{L^{2q}}+\left\||u_t(s,\cdot)|^{q-1}\nabla u_t(s,\cdot)\right\|_{L^2}\right)\nonumber\\
&\leqslant C(1+t-s)^{-\frac{1}{2}}\left(\|u_t(s,\cdot)\|^q_{L^{2q}}+\||u_t(s,\cdot)|^{q-1}\|_{L^{2q/(q-1)}}\|\nabla u_t(s,\cdot)\|_{L^{2q}}\right)\nonumber\\
&= C(1+t-s)^{-\frac{1}{2}}\left(\|u_t(s,\cdot)\|^q_{L^{2q}}+\|u_t(s,\cdot)\|^{q-1}_{L^{2q}}\|\nabla u_t(s,\cdot)\|_{L^{2q}}\right)\nonumber\\
&\doteq C\left(K_1(t,s)+K_2(t,s)\right),
\end{align}
where 
\begin{align*}
K_1(t,s)&\doteq (1+t-s)^{-\frac{1}{2}}\|u_t(s,\cdot)\|^q_{L^{2q}},\\
K_2(t,s)&\doteq (1+t-s)^{-\frac{1}{2}}\|u_t(s,\cdot)\|^{q-1}_{L^{2q}}\|\nabla u_t(s,\cdot)\|_{L^{2q}}.
\end{align*}
Again, by  $u(t,\cdot)\in H^2(\Omega)$ for $t\in[0,T]$, the Gagliardo-Nirenberg inequality can be applied to get
\begin{align*}
\|u_t(s,\cdot)\|_{L^{2q}}&\leqslant C\|u_t(s,\cdot)\|^{1-\theta_2(2q)}_{L^2}\|\nabla u_t(s,\cdot)\|^{\theta_2(2q)}_{L^2},\\
\|\nabla u_t(s,\cdot)\|_{L^{2q}}&\leqslant C\|\nabla u_t(s,\cdot)\|^{1-\theta_2(2q)}_{L^2}\|\Delta u_t(s,\cdot)\|^{\theta_2(2q)}_{L^2},
\end{align*}
where $\theta_2(2q):=\theta(2q)=2(\frac{1}{2}-\frac{1}{2q})=1-\frac{1}{q}$.\\
 Due to the fact that $\psi(s,x)>0$ for all $x\in\Omega$ and $s\in[0,t]$, the next two estimates hold:
\begin{align}\label{est31}
\|u_t(s,\cdot)\|_{L^{2q}}&\leqslant C(1+s)^{-\frac{1}{2}}W[u](s)^{\frac{1}{2}},\\
\|\nabla u_t(s,\cdot)\|_{L^{2q}}&\leqslant C(1+s)^{-\frac{1-\theta_2(2q)}{2}}W[u](s)^{\frac{1-\theta_2(2q)}{2}}\left\|\ee^{\psi(s,\cdot)}\Delta u_t(s,\cdot)\right\|_{L^2}^{\theta_2(2q)}.\label{est32}
\end{align}

Therefore, the combination of \eqref{est30} and \eqref{est31} yields
\begin{align}\label{est33}
\int_0^tK_1(t,s)\,\dd s&\leqslant C\int_0^t(1+t-s)^{-\frac{1}{2}}(1+s)^{-\frac{q}{2}}W[u](s)^{\frac{q}{2}}\,\dd s\nonumber\\
&\leqslant CM[u](t)^{\frac{q}{2}}\int_0^t(1+t-s)^{-\frac{1}{2}}(1+s)^{-\frac{q}{2}}\,\dd s\nonumber\\
&\leqslant C(1+t)^{-\frac{1}{2}}M[u](t)^{\frac{q}{2}},
\end{align}
where Lemma 4.1 in \cite{Cui2001} has been applied again with $q>2$.\\
To estimate another term with respect to $K_2(t,s)$, using \eqref{est31} and \eqref{est32} one may obtain
\begin{align}\label{est34}
\int_0^tK_2(t,s)\,\dd s& \leqslant  C\int_0^t(1+t-s)^{-\frac{1}{2}}(1+s)^{-\frac{q-\theta_2(2q)}{2}}W[u](s)^{\frac{q-\theta_2(2q)}{2}}\left\|\ee^{\psi(s,\cdot)}\Delta u_t(s,\cdot)\right\|^{\theta_2(2q)}_{L^2}\dd s\nonumber\\
&\leqslant C\left(\int_0^t(1+t-s)^{-\frac{1}{2-\theta_2(2q)}}(1+s)^{-\frac{q-\theta_2(2q)}{2-\theta_2(2q)}}W[u](s)^{\frac{q-\theta_2(2q)}{2-\theta_2(2q)}}\,\dd s\right)^{(2-\theta_2(2q))/2}\nonumber\\
&\qquad\times\left(\int_0^t\left\|\ee^{\psi(s,\cdot)}\Delta u_t(s,\cdot)\right\|^{2}_{L^2}\dd s\right)^{\theta_2(2q)/2}\nonumber\\
&\leqslant CM[u](t)^{\frac{q-\theta_2(2q)}{2}}\left(\int_0^t(1+t-s)^{-\frac{1}{2-\theta_2(2q)}}(1+s)^{-\frac{q-\theta_2(2q)}{2-\theta_2(2q)}}\,\dd s\right)^{(2-\theta_2(2q))/2}\nonumber\\
&\qquad\times\left(\int_0^t\left\|\ee^{\psi(s,\cdot)}\Delta u_t(s,\cdot)\right\|^{2}_{L^2}\dd s\right)^{\theta_2(2q)/2}\nonumber\\
&\leqslant C(1+t)^{-\frac{1}{2}}M[u](t)^{\frac{q-\theta_2(2q)}{2}}\left(\int_0^t\left\|\ee^{\psi(s,\cdot)}\Delta u_t(s,\cdot)\right\|^{2}_{L^2}\dd s\right)^{\theta_2(2q)/2},
\end{align}
where H\"older's inequality is employed since $\theta_2(2q)<1$ for $q>2$. Note that here we have also used Lemma 4.1 in \cite{Cui2001}, because
\begin{align*}
\frac{1}{2-\theta_2(2q)}<1\quad\mbox{and}\quad  \frac{q-\theta_2(2q)}{2-\theta_2(2q)}>1.
\end{align*}

 Moreover, Proposition \ref{prop1} and \eqref{est16} tell us the integration of higher-order weighted energy can be controlled by the following way:
 \begin{align*}
 \int_0^t\left\|\ee^{\psi(t,\cdot)}\Delta u_t(s,\cdot)\right\|_{L^2}^2\dd s&\leqslant \left\|\ee^{\psi(0,\cdot)}\nabla u_1\right\|_{L^2}^2+\left\|\ee^{\psi(0,\cdot)}\Delta u_0\right\|^2_{L^2}+C\int_0^t\left\|\ee^{\psi(s,\cdot)}F(u)(s,\cdot)\right\|_{L^2}^2\dd s\\
 &\leqslant I_{\mathrm{exp}}[u_0,u_1]+ C\left(M[u](t)^p+M[u](t)^q\right).
 \end{align*}
All in all, we may conclude
\begin{align}\label{est35}
\int_0^tK_2(t,s)\,\dd s&\leqslant C(1+t)^{-\frac{1}{2}}M[u](t)^{\frac{q-\theta_2(2q)}{2}}\notag\\
&\quad\times\left(I_{\mathrm{exp}}[u_0,u_1]^{\frac{\theta_2(2q)}{2}}+M[u](t)^{\frac{p\theta_2(2q)}{2}}+M[u](t)^{\frac{q\theta_2(2q)}{2}}\right).
\end{align}
Using the derived estimates \eqref{est33} and \eqref{est35}, we claim that
\begin{align}\label{est36}
&(1+t)^{\frac{1}{2}}\|\ml{D}_2u_{N_2}(t,\cdot)\|_{L^2}\notag\\
&\leqslant C(1+t)^{\frac{1}{2}}\int_0^t\left\|\ml{D}_2E(t-s,\cdot)\ast_{(\cdot)}|u_t(s,\cdot)|^q\right\|_{L^2}\dd s\notag\\
&\leqslant C(1+t)^{\frac{1}{2}}\left( \int_0^tK_1(t,s)\,\dd s+\int_0^tK_2(t,s)\,\dd s\right)\nonumber\\
&\leqslant CM[u](t)^{\frac{q}{2}}+M[u](t)^{\frac{q-\theta_2(2q)}{2}}\left(I_{\mathrm{exp}}[u_0,u_1]^{\frac{\theta_2(2q)}{2}}+M[u](t)^{\frac{p\theta_2(2q)}{2}}+M[u](t)^{\frac{q\theta_2(2q)}{2}}\right).
\end{align}

It remains to estimate $\|\ml{D}^j_1u_{N_2}(t,\cdot)\|_{L^2}$ for $j=0,1$. Let us apply Theorem 2.1 in \cite{Ikehata-Inoue2008} to have
\begin{align*}
   \left\|\ml{D}^j_1E_1(t-s,\cdot)\ast_{(\cdot)}|u_t(s,\cdot)|^q\right\|_{L^2}&\leqslant C(1+t-s)^{-\frac{j}{2}}\left(\left\|u_t(s,\cdot)\right\|_{L^{2q}}^q+\left\|d(\cdot) |u_t(s,\cdot)|^q\right\|_{L^2}\right).
\end{align*}
Using Lemma 2.5 in \cite{Ikehata-Inoue2008} and replacing $u$ by $u_t$, one derives
\begin{align*}
\left\|d(\cdot) |u_t(s,\cdot)|^q\right\|_{L^2}\leqslant C(1+s)^{\frac{(2+\rho)(1+\varepsilon_1)}{2}}\left\| \ee^{\delta\psi(s,\cdot)}u_t(s,\cdot)\right\|^q_{L^{2q}}
\end{align*}
for any $\varepsilon_1>0$, $\rho>0$ and $\delta>0$. It is obvious that
\begin{align*}
\|u_t(s,\cdot)\|_{L^{2q}}^q\leqslant C(1+s)^{\frac{(2+\rho)(1+\varepsilon_1)}{2}}\left\| \ee^{\delta\psi(s,\cdot)}u_t(s,\cdot)\right\|^q_{L^{2q}},
\end{align*}
which implies that
\begin{align*}
\left\|\ml{D}^j_1E_1(t-s,\cdot)\ast_{(\cdot)}|u_t(s,\cdot)|^q\right\|_{L^2}\leqslant C(1+s)^{\frac{(2+\rho)(1+\varepsilon_1)}{2}}\left\| \ee^{\delta\psi(s,\cdot)}u_t(s,\cdot)\right\|^q_{L^{2q}}.
\end{align*}
Eventually, the following estimate holds:
\begin{align*}
	\left\|\ml{D}^j_1u_{N_2}(t,\cdot)\right\|_{L^2}&\leqslant C\int_0^t\left\|\ml{D}^j_1E_1(t-s,\cdot)\ast_{(\cdot)}|u_t(s,\cdot)|^q\right\|_{L^2}\dd s\nonumber\\
	&\leqslant C\int_0^t(1+t-s)^{-\frac{j}{2}}(1+s)^{\frac{(2+\rho)(1+\varepsilon_1)}{2}}\left\| \ee^{\delta\psi(s,\cdot)}u_t(s,\cdot)\right\|^q_{L^{2q}}\dd s \nonumber\\
	&\leqslant  C\int_0^t(1+t-s)^{-\frac{j}{2}}(1+s)^{-(1+\eta)}\left((1+s)^{\frac{(2+\rho)(1+\varepsilon_1)}{2q}+\frac{1+\eta}{q}}\left\| \ee^{\delta\psi(s,\cdot)}u_t(s,\cdot)\right\|_{L^{2q}}\right)^q\dd s \nonumber\\
	&\leqslant  C(1+t)^{-\frac{j}{2}}\left(\sup\limits_{s\in[0,t]}(1+s)^{\beta_5}\left\| \ee^{\delta\psi(s,\cdot)}u_t(s,\cdot)\right\|_{L^{2q}}\right)^q,
\end{align*}
where the constant is denoted by
\begin{align*}
\beta_5\doteq\frac{(2+\rho)(1+\varepsilon_1)}{2q}+\frac{1+\eta}{q}.
\end{align*}
 By Lemma 2.3 from \cite{Ikehata-Inoue2008}, we get for $j=0,1$ that
\begin{align*}
	\left\|\ml{D}^j_1u_{N_2}(t,\cdot)\right\|_{L^2}&\leqslant C(1+t)^{-\frac{j}{2}}\left(\sup\limits_{s\in[0,t]}(1+s)^{\beta_5}(1+s)^{\frac{(2+\rho)(1-\theta_2(2q))}{2}}\| \nabla u_t(s,\cdot)\|^{1-\delta}_{L^2}\left\| \ee^{\psi(s,\cdot)}\nabla u_t(s,\cdot)\right\|_{L^2}^\delta\right)^q\nonumber\\
	&= C(1+t)^{-\frac{j}{2}}\left(\sup_{s\in[0,t]}(1+s)^{\beta_6} \left((1+s)^{\frac{1}{2}}\|\nabla u_t(s,\cdot)\|_{L^2}\right)^{1-\delta}\left\| \ee^{\psi(s,\cdot)}\nabla u_t(s,\cdot)\right\|_{L^2}^\delta\right)^q\nonumber\\
	&\leqslant  C(1+t)^{-\frac{j}{2}}\left(\sup_{s\in[0,t]}(1+s)^{\beta_6}W[u](s)^{\frac{1}{2}}\right)^q
\end{align*}
where the constant is defined by
\begin{align*}
\beta_6\doteq\beta_5+\frac{(2+\rho)(1-\theta_2(2q))}{2}-\frac{1-\delta}{2}=\frac{\varepsilon_1(2+\rho)+2\eta}{2q}+\frac{\delta}{2}-\left(\frac{1}{2}-\frac{6+2\rho}{2q}\right)<0
\end{align*}
if $q>6+2\rho$ by taking sufficiently small constants $\varepsilon_1$, $\eta$ and $\delta$.\\
In conclusion, we derive
\begin{equation}\label{est37}
\left\|\ml{D}^j_1u_{N_2}(t,\cdot)\right\|_{L^2}\leqslant C(1+t)^{-\frac{j}{2}}M[u](t)^{\frac{q}{2}}.
\end{equation}

Finally, summarizing the derive estimates \eqref{est25}, \eqref{est29}, \eqref{est36} and \eqref{est37} we conclude
\begin{align*}
\|u^{\non}(t,\cdot)\|_{L^2}^2+(1+t)\|\ml{D}u^{\non}(t,\cdot)\|_{L^2}^2\leqslant C\left(M[u](t)^{p}+M[u](t)^{q}+M[u](t)^{q+\frac{1}{q}-1}\widetilde{M}[u](t;p,q)\right),
\end{align*}
where we would like to show again that
\begin{align*}
\widetilde{M}[u](t;p,q)= I_{\mathrm{exp}}[u_0,u_1]^{\frac{q-1}{q}}+M[u](t)^{\frac{p(q-1)}{q}}+M[u](t)^{q-1}.
\end{align*}
Then, we derive our desired estimate \eqref{Important.01}. Namely, we can claim that $u\in X(T)$ and $N$ maps $X(T)$ into itself.

To prove the Lipschitz condition \eqref{Important.02}, we may compute
\begin{align*}
M[Nu-Nv](T)&=M\left[\int_0^tE_1(t-s,x)\ast_{(x)}\left(F(u)(s,x)-F(v)(s,x)\right)\dd s\right](T)\\
&\doteq M[Nw](T),
\end{align*}
where
\begin{align*}
w(t,x)\doteq u(t,x)-v(t,x).
\end{align*}
In other words, we need to estimate the next four norms:
\begin{align*}
\left\|\ee^{\psi(t,\cdot)}\ml{D}Nw(t,\cdot)\right\|_{L^2}^2,\quad \left\|\ml{D}_1Nw(t,\cdot)\right\|^2_{L^2},\quad \left\|\ml{D}_2Nw(t,\cdot)\right\|^2_{L^2}\quad\text{and}\quad \left\|Nw(t,\cdot)\right\|^2_{L^2}.
\end{align*}

First of all, by applying
\begin{align}\label{Young}
\left||u(s,x)|^p-|v(s,x)|^p\right|\leqslant C|w(s,x)|\left(|u(s,x)|^{p-1}+|v(s,x)|^{p-1}\right),
\end{align}
and H\"older's inequality, one has the estimate for solution in the weighted $L^2$-norm
\begin{align*}
&\left\|\ee^{\psi(s,\cdot)}\left(|u(s,\cdot)|^p-|v(s,\cdot)|^p\right)\right\|_{L^2}\\
&\leqslant C\left\|\ee^{\psi(s,\cdot)}|w(s,\cdot)|\left(|u(s,\cdot)|^{p-1}+|v(s,\cdot)|^{p-1}\right)\right\|_{L^2}\\
&\leqslant C\left\|\ee^{\frac{1}{p}\psi(s,\cdot)}w(s,\cdot)\right\|_{L^{2p}}\left(\left\|\ee^{\frac{1}{p}\psi(s,\cdot)}u(s,\cdot)\right\|_{L^{2p}}^{p-1}+\left\|\ee^{\frac{1}{p}\psi(s,\cdot)}v(s,\cdot)\right\|_{L^{2p}}^{p-1}\right),
\end{align*}
and the estimate for solution in the $L^2$-norm
\begin{align*}
\left\||u(s,\cdot)|^p-|v(s,\cdot)|^p\right\|_{L^2}\leqslant C\|w(s,\cdot)\|_{L^{2p}}\left(\|u(s,\cdot)\|_{L^{2p}}^{p-1}+\|v(s,\cdot)\|_{L^{2p}}^{p-1}\right).
\end{align*}
Next, repeating the same procedure as the proof of \eqref{Important.01}, we may derive
\begin{align}\label{Imp03}
\left\|\ee^{\psi(t,\cdot)}\ml{D}Nw(t,\cdot)\right\|_{L^2}^2\leqslant C M[w](t)\sum\limits_{r=p-1,q-1,(p-1)/2,(q-1)/2}\left(M[u](t)^r+M[v](t)^r\right),
\end{align}
\begin{align}\label{Imp04}
\|Nw(t,\cdot)\|_{L^2}^2+(1+t)\|\ml{D}_1Nw(t,\cdot)\|_{L^2}^2\leqslant C M[w](t)\sum\limits_{r=p-1,q-1}\left(M[u](t)^r+M[v](t)^r\right),
\end{align}
providing that our assumptions $p,q>6+2\rho_0$ hold.

Then, to conclude the remaining part of the Lipschitz condition, we just need the estimate of 
\begin{align*}
L(t)\doteq\left\|\nabla(F(u)(s,\cdot)-F(v)(s,\cdot))\right\|_{L^2}.
\end{align*}
We divide the proof by two parts
\begin{align*}
L_1(t)&\doteq\left\|\nabla\left(|u(s,\cdot)|^p-|v(s,\cdot)|^p\right)\right\|_{L^2},\\
L_2(t)&\doteq\left\|\nabla\left(|u_t(s,\cdot)|^q-|v_t(s,\cdot)|^q\right)\right\|_{L^2}.
\end{align*}
Obviously, it follows $L(t)\leqslant C L_1(t)+CL_2(t)$.\\
Let us sketch the proof due to the fact the proof is standard (see, for example, \cite{Palmieri-Reissig,Djaouti-Reissig2018,Chen-Reissig2019}). Setting $g(f)=f|f|^{p-2}$, we may rewrite the different of nonlinearity by
\begin{align*}
|u(s,x)|^p-|v(s,x)|^p=p\int_0^1w(s,x)g\left(\nu u(s,x)+(1-\nu)v(s,x)\right)\dd \nu.
\end{align*}
Consequently, some applications of Minkowski's inequality and the Leibniz rule show that
\begin{align*}
\left\|\nabla \left(|u(s,\cdot)|^p-|v(s,\cdot)|^p\right)\right\|_{L^2}&\leqslant C\int_0^1\|\nabla w(s,\cdot)\|_{L^{r_1}}\|g\left(\nu u(s,\cdot)+(1-\nu)v(s,\cdot)\right)\|_{L^{r_2}}\dd\nu\\
&\quad+C\int_0^1\|w(s,\cdot)\|_{L^{r_3}}\|\nabla g\left(\nu u(s,\cdot)+(1-\nu)v(s,\cdot)\right)\|_{L^{r_4}}\dd\nu,
\end{align*}
where $1/r_1+1/r_2=1/r_3+1/r_4=1/2$ and these parameters will be determined later.\\
We notice from the Gagliardo-Nirenberg inequality that
\begin{align*}
\int_0^1\|g\left(\nu u(s,\cdot)+(1-\nu)v(s,\cdot)\right)\|_{L^{r_2}}\dd\nu&\leqslant C\left(\|u(s,\cdot)\|_{L^{r_2(p-1)}}^{p-1}+\|v(s,\cdot)\|_{L^{r_2(p-1)}}^{p-1}\right),\\
\|u(s,\cdot)\|_{L^{r_2(p-1)}}&\leqslant C (1+s)^{-\frac{\theta(r_2(p-1))}{2}}W[u](s)^{\frac{1}{2}},\\
\|v(s,\cdot)\|_{L^{r_2(p-1)}}&\leqslant C (1+s)^{-\frac{\theta(r_2(p-1))}{2}}W[v](s)^{\frac{1}{2}},\\
\|\nabla w(s,\cdot)\|_{L^{r_1}}&\leqslant C (1+s)^{-\frac{1}{2}}W[w](s)^{\frac{1}{2}},\\
\|w(s,\cdot)\|_{L^{r_3}}&\leqslant C(1+s)^{-\frac{\theta(r_3)}{2}}W[w](s)^{\frac{1}{2}},
\end{align*}
where $\theta(r_2(p-1))=1-1/(r_2(p-1))$ and $\theta(r_3)=1-2/r_3$.\\
Furthermore, the applications of the chain rule and the Gagliardo-Nirenberg inequality imply
\begin{align*}
&\|\nabla g\left(\nu u(s,\cdot)+(1-\nu)v(s,\cdot)\right)\|_{L^{r_4}}\\
&\leqslant C\|\nu u(s,\cdot)+(1-\nu)v(s,\cdot)\|_{L^{r_5}}^{p-2}\|\nabla(\nu u(s,\cdot)+(1-\nu)v(s,\cdot))\|_{L^{r_6}}\\
&\leqslant C\left(\|u(s,\cdot)\|_{L^{r_5}}+\|v(s,\cdot)\|_{L^{r_5}}\right)^{p-2}\left(\|\nabla u(s,\cdot)\|_{L^{r_6}}+\|\nabla v(s,\cdot)\|_{L^{r_6}}\right)\\
&\leqslant C(1+s)^{-\frac{\theta(r_5)(p-2)+1}{2}}\left(W[u](s)^{\frac{p-1}{2}}+W[v](s)^{\frac{p-1}{2}}\right)
\end{align*}
with $1/r_4=(p-2)/r_5+ 1/r_6$ and $\theta(r_5)=1-2/r_5$.\\
Combining with the above derived estimates, we obtain
\begin{align*}
L_1(t)\leqslant C \left((1+s)^{d_1}+(1+s)^{d_2}\right)W[w](s)^{\frac{1}{2}}\left(W[u](s)^{\frac{p-1}{2}}+W[v](s)^{\frac{p-1}{2}}\right),
\end{align*}
where the parameters in the estimate are defined by
\begin{align*}
d_1\doteq-\frac{p}{2}+\frac{1}{r_2}\quad\mbox{and}\quad d_2\doteq\frac{1}{r_3}+\frac{p-2}{r_5}-\frac{p}{2}.
\end{align*}
By choosing
\begin{align*}
r_1=r_4=\frac{1}{\varepsilon_2},\quad r_2=r_3=\frac{2}{1-2\varepsilon_2},\quad r_5=r_6=\frac{2(p-2)}{\varepsilon_2}
\end{align*}
with sufficiently small constant $\varepsilon_2\rightarrow0^+$, we can show when $p>3$, the constants satisfy
\begin{align*}
d_1<-1\quad\mbox{and}\quad d_2<-1.
\end{align*}
By using the same approach of the above, we get
\begin{align*}
L_2(t)&\leqslant C (1+s)^{-\frac{2q-1}{4}}W[w](s)^{\frac{1}{4}}\left(W[u](s)^{\frac{1}{2}}+W[v](s)^{\frac{1}{2}}\right)\left\|\ee^{\psi(s,\cdot)}\Delta w_t(s,\cdot)\right\|_{L^2}^{\frac{1}{2}}\\
&\quad+C(1+s)^{-\frac{2q-1}{4}}W[w](s)^{\frac{1}{2}}\widetilde{\widetilde{W}}[u,v](s;q)\left(\left\|\ee^{\psi(s,\cdot)}\Delta u_t(s,\cdot)\right\|_{L^2}^{\frac{1}{2}}+\left\|\ee^{\psi(s,\cdot)}\Delta v_t(s,\cdot)\right\|_{L^2}^{\frac{1}{2}}\right),
\end{align*}
where
\begin{align*}
\widetilde{\widetilde{W}}[u,v](t;q)\doteq\left(W[u](t)^{\frac{1}{2}}+W[v](t)^{\frac{1}{2}}\right)^{q-2}\left(W[u](t)^{\frac{1}{4}}+W[v](t)^{\frac{1}{4}}\right).
\end{align*}
Finally, we conclude
\begin{align}\label{Imp05}
\left\|\ml{D}_2Nw(t,\cdot)\right\|^2_{L^2}&\leqslant C M[w](t)\left(M[u](t)^{p-1}+M[v](t)^{p-1}\right)\notag\\
&\quad+CM[w](t)\left(M[w](t)^{p-1}+M[w](t)^{q-1}\right)^{\frac{1}{2}}\left(M[u](t)+M[v](t)\right)\notag\\
&\quad+CM[w](t)\widetilde{\widetilde{M}}[u,v](t)\left(M[u](t)^p+M[v](t)^q\right)^{\frac{1}{2}}.
\end{align}
Summarizing the derived estimates \eqref{Imp03}, \eqref{Imp04} and \eqref{Imp05}, we claim \eqref{Important.02} holds. 

Applying the Banach fixed-point theorem, our proof is complete.

\section{Final remark}\label{Sec.Final.Remark}
\setcounter{equation}{0}
In this paper, we prove global (in time) existence of small data solution 
\begin{align*}
u\in\ml{C}\left([0,\infty),H^2(\Omega)\cap H^1_0(\Omega)\right)\cap\ml{C}^1\left([0,\infty),H^1(\Omega)\right)
\end{align*}
to the exterior problem \eqref{Eq.Semi.BVP} with $p,q>6+2\rho_0$. Let us give some  explanations for the conditions of $p$ and $q$. Actually, we may observe from the proof that the condition for the exponent $p$ is influenced by the value of $\rho$. Because we apply the weighted function $\ee^{\psi(t,x)}$ with parameter $\rho$ on the energy estimates in this paper, the interplay between the power nonlinearity $|u|^p$ and $|u_t|^q$ comes. Moreover, to control the higher-order energy, we choose a suitable parameter $\rho$ in the weighted function such that $\rho>\rho_0$.

To end the paper, we give some remarks on the semilinear strongly damped wave equations with an exterior domain for higher-dimensional case, namely,
\begin{equation}\label{Eq.Semi.BVP.higher}
\begin{cases}
u_{tt}-\Delta u -\Delta u_t =f(u,u_t;p,q), &x\in \Omega,\,t>0,\\
u(0,x)=  u_0(x),\,\,u_t(0,x)=  u_1(x),&x\in \Omega,\\
u=0,&x\in \partial\Omega,\,t>0,
\end{cases}
\end{equation} 
with $p,q>1$, where $\Omega\subset\mathbb{R}^n$ for $n\geqslant 3$ is an exterior domain with a compact smooth boundary $\partial\Omega$. Without loss of generality, we assume again that $0\notin\overline{\Omega}$. In \eqref{Eq.Semi.BVP.higher}, the nonlinearity can be represented by
\begin{align*}
f(u,u_t;p,q)\doteq a|u|^p+b|u_t|^q
\end{align*}
with $a,b\geqslant0$ but $a+b\neq0$, and $p,q>1$. From \cite{Ikehata-Inoue2008}, we understand the difficulties to study global (in time) existence of small data solutions in higher-dimension ($n\geqslant 3$) is that we must restrict the power $p$ to the range $1<p,q\leqslant n/(n-2)$, which is restricted by the application of the Gagliardo-Nirenberg inequality or Sobolev inequality. Nevertheless, by applying the weighted energy method, we always proposed the a strong condition such as $p>6$ for $f(u,u_t;p,q)=|u|^p$ in \cite{Ikehata-Inoue2008}, and $p,q>6+2\rho_0$ for $f(u,u_t;p,q)=|u|^p+|u_t|^q$ in this paper for 2D. It will immediately leads to the empty range of $p$ and $q$. To solve this difficulty, motivated by the main approach of this paper, we observe that there exists a possibility to consider higher-order energy solutions with even large regular data. At this time, we may apply the embedding $H^s(\Omega)\hookrightarrow L^{\infty}(\Omega)$ for $s>n/2$ rather than apply the Gagliardo-Nirenberg inequality for estimating the solutions in the $L^{2p}$ and $L^{2q}$ norms. Furthermore, the benefit of this approach is to weaken the upper bound restrictions of $p,q$ from $n/(n-2)$ to $\infty$ (One may see this effect in Section 6.2 of \cite{Palmieri-Reissig}). Thus, we just need to derive higher-order energy estimates with large regular data. But, we should emphasize that higher-order energy estimates with exponentially weighted function are still open.

\appendix
\section{Gagliardo-Nirenberg type inequalities}
Let us introduce a well-known interpolation inequality, i.e., the Gagliardo-Nirenberg inequality in exterior domains in 2D. This result has been proved by \cite{Crispo-Marmonti2004}.
\begin{lem}
Let $1\leqslant r\leqslant q\leqslant +\infty$. If $v\in H^1(\Omega)$ with a exterior domain $\Omega\subset\mb{R}^2$ with compact boundary, having the cone property, then the inequality holds
\begin{align*}
\|v\|_{L^q}\leqslant M\|v\|_{L^2}^{1-\theta(q)}\|\nabla v\|_{L^2}^{\theta(q)},
\end{align*}
where $M>0$ is a constant independent of $v$ and $\theta(q)=2(1/2-1/q)\in(0,1]$.
\end{lem}

Next, we give a weighted version of Gagliardo-Nirenberg inequality in exterior domain. For $\sigma>0$, $\rho>0$ and $t\geqslant0$, we may define a family of weighted function space by
\begin{align*}
H^1_{\sigma\psi(t,\cdot)}(\Omega)\doteq\left\{f\in H^1(\Omega):\,\left\|\ee^{\sigma\psi(t,\cdot)}f\right\|_{L^2}^2+\left\|\ee^{\sigma\psi(t,\cdot)}\nabla f\right\|_{L^2}^2<\infty\,\,\mbox{for}\,\,t\in[0,T]\right\}.
\end{align*}

According to our choice of weighted function $\psi(t,x)$, the following Gagliardo-Nirenberg type inequality holds (cf. Lemma 2.3 in \cite{Ikehata-Inoue2008})
\begin{lem}\label{Lem.G-N}
Let $\theta(q)=2(1/2-1/q)$ and $0\leqslant\theta(q)<1$ and let $0<\sigma\leqslant 1$, $\rho>0$. If $v\in H^{1}_{\psi(t,\cdot)}(\Omega)$ with $t\geqslant 0$, then the inequality holds
\begin{align*}
\left\|\ee^{\sigma\psi(t,\cdot)}v\right\|_{L^q}\leqslant C(1+t)^{\frac{(2+\rho)(1-\theta(q))}{2}}\|\nabla v\|_{L^2}^{1-\sigma}\left\|\ee^{\psi(t,\cdot)}\nabla v\right\|_{L^2}^{\sigma}
\end{align*}
for each $t\geqslant 0$, where $C=C_{\sigma}>0$ is a positive constant.
\end{lem}

\subsection*{Acknowledgments} The Ph.D. study of Wenhui Chen is supported by S\"achsiches Landesgraduiertenstipendium. The authors thank Professor Ryo Ikehata for the suggestions in the preparation of the paper.



\begin{thebibliography}{99}
\bibitem{Cui2001} S. Cui. Local and global existence of solutions to semilinear parabolic initial value problems. \emph{Nonlinear Anal.} \textbf{43} (2001), no. 3, 293--323.
\bibitem{Chen-Fino2019} W. Chen, A.Z. Fino. Blow-up of solutions to semilinear strongly damped wave equations with different nonlinear terms in an exterior domain. \emph{Preprint} arXiv:1910.05981.
\bibitem{Chen-Reissig2019} W. Chen, M. Reissig. Weakly coupled systems of semilinear elastic waves with different damping mechanisms in 3D. \emph{Math. Methods Appl. Sci.} \textbf{42} (2019), no. 2, 667--709.
\bibitem{Crispo-Marmonti2004} F. Crispo, P. Maremonti. An interpolation inequality in exterior domains. \emph{Rend. Sem. Mat. Univ. Padova} \textbf{112} (2004), 11--39.
\bibitem{Dabbicco-Takeda-Ikehata2019} M. D'Abbicco, H. Takeda, R. Ikehata. Critical exponent for semi-linear wave equations with double damping terms in exterior domains. \emph{Accepted by NoDEA Nonlinear Differential Equations Appl.} (2019).
\bibitem{DAbbiccoReissig2014} M. D'Abbicco, M. Reissig. Semilinear structural damped waves. \emph{Math. Methods Appl. Sci.} \textbf{37} (2014), no. 11, 1570--1592.
\bibitem{D'AmbrosioLucente2003} L. D'Ambrosio, S. Lucente. Nonlinear Liouville theorems for Grushin and Tricomi operators. \emph{J. Differential Equations} \textbf{193} (2003), no. 2, 511--541.
\bibitem{Fino} A.Z. Fino. Finite time blow up for wave equations with strong damping in an exterior domain. \emph{Preprint} arXiv: 2695271. 
\bibitem{Fino-Ibrahim-Wehbe2017} A.Z. Fino, H. Ibrahim, A. Wehbe. A blow-up result for a nonlinear damped wave equation in exterior domain: the critical case. \emph{Comput. Math. Appl.} \textbf{73} (2017), no. 11, 2415--2420.
\bibitem{Hayashi-Kaikina-Naumkin2004} N. Hayashi, E.I. Kaikina, P.I. Naumkin. Damped wave equation with a critical nonlinearity on a half line. \emph{J. Anal. Appl.} \textbf{2} (2004), no. 2, 95--112.
\bibitem{Ikehata2014} R. Ikehata. Asymptotic profiles for wave equations with strong damping. \emph{J. Differential Equations} \textbf{257} (2014), no. 6, 2159--2177.
\bibitem{Ikehata2004} R. Ikehata. Global existence of solutions for semilinear damped wave equation in 2-D exterior domain. \emph{J. Differential Equations} \textbf{200} (2004), no. 1, 53--68.
\bibitem{Ikehata2003} R. Ikehata. Critical exponent for semilinear damped wave equations in the N-dimensional half space. \emph{J. Math. Anal. Appl.} \textbf{288} (2003), no. 2, 803--818.
\bibitem{Ikehata2001} R. Ikehata. Decay estimates of solutions for the wave equations with strong damping terms in unbounded domains. \emph{Math. Methods Appl. Sci.} \textbf{24} (2001), no. 9, 659--670.
\bibitem{Ikehata-Inoue2008} R. Ikehata, Y. Inoue. Global existence of weak solutions for two-dimensional semilinear wave equations with strong damping in an exterior domain. \emph{Nonlinear Anal.} \textbf{68} (2008), no. 1, 154--169.
\bibitem{Ikehata-Tanizawa2005} R. Ikehata, K. Tanizawa. Global existence of solutions for semilinear damped wave equations in $\mathbf{R}^N$ with non compactly supported initial data. \emph{Nonlinear Anal.} \textbf{61} (2005), no. 7, 1189--1208.
\bibitem{IkehataTodorovaYordanov2013} R. Ikehata, G. Todorova, B. Yordanov. Wave equations with strong damping in Hilbert spaces. \emph{J. Differential Equations} \textbf{254} (2013), no. 8, 3352--3368.
\bibitem{Lai-Yin2017} N. Lai, S. Yin. Finite time blow-up for a kind of initial-boundary value problem of semilinear damped wave equation. \emph{Math. Methods Appl. Sci.} \textbf{40} (2017), no. 4, 1223--1230.
\bibitem{Djaouti-Reissig2018} A. Mohammed Djouti, M. Reissig. Weakly coupled systems of semilinear effectively damped waves with time-dependent coefficient, different power nonlinearities and different regularity of the data. \emph{Nonlinear Anal.} \textbf{175} (2018), 28--55.
\bibitem{One2003} K. Ono. Decay estimates for dissipative wave equations in exterior domains. \emph{J. Math. Anal. Appl.} \textbf{286} (2003), no. 2, 540--562.
\bibitem{Ogawa-Takeda2009} T. Ogawa, H. Takeda. Non-existence of weak solutions to nonlinear damped wave equations in exterior domains. \emph{Nonlinear Anal.} \textbf{70} (2009), no. 10, 3696--3701.
\bibitem{Palmieri-Reissig} A. Palmieri, M. Reissig. Semi-linear wave models with power non-linearity and scale-invariant time-dependent mass and dissipation, II. \emph{Math. Nachr.} \textbf{291} (2018), no. 11-12, 1859--1892.
\bibitem{Ponce1985} G. Ponce. Global existence of small solutions to a class of nonlinear evolution equations. \emph{Nonlinear Anal.} \textbf{9} (1985), no. 5, 399--418.
\bibitem{Shibata2000} Y. Shibata.  On the rate of decay of solutions to linear viscoelastic equation. \emph{Math. Methods Appl. Sci.} \textbf{23} (2000), no. 3, 203--226.
\bibitem{Sobajima2019} M. Sobajima. Global existence of solutions to semilinear damped wave equation with slowly decaying initial data in exterior domain. \emph{Differential Integral Equations} \textbf{32} (2019), no. 11-12, 615--638.
\bibitem{Sobajima-Wakasugi2019} M. Sobajima, Y. Wakasugi.  Weighted energy estimates for wave equation with space-dependent damping term for slowly decaying initial data. \emph{Commun. Contemp. Math.} \textbf{21} (2019), no. 5, 30 pp.
\bibitem{Todorova-Yordanov2001} G. Todorova, B. Yordanov. Critical exponent for a nonlinear wave equation with damping. \emph{J. Differential Equations} \textbf{174} (2001), no. 2, 464--489.
\end{thebibliography}
\end{document}